\RequirePackage{fix-cm}
\documentclass[smallextended]{svjour3}       

\usepackage{graphicx}%
\usepackage{booktabs}
\usepackage{multirow}%
\usepackage{amsmath,amssymb,amsfonts}%
\usepackage{mathrsfs}%
\usepackage[title]{appendix}%
\usepackage{xcolor}
\usepackage{manyfoot}%
\usepackage{booktabs}%
\usepackage{algorithm}%
\usepackage{algorithmicx}%
\usepackage{algpseudocode}%
\usepackage{listings}%
\usepackage[justification=centering]{caption}
\usepackage{hyperref}
\hypersetup{hidelinks}
\usepackage{geometry}
\usepackage{nicefrac} 
\usepackage[final]{showlabels}

\usepackage{changes}
\definechangesauthor[name={Chao Ding}, color=blue]{Chao}
\definechangesauthor[name={FXY Feng}, color=green]{Feng}
\newcommand{\vio}{\color{black}}      
\newcommand{\tvio}{\color{black}}
\usepackage{float}
 
\usepackage{soul}%

\newcommand{\lam}{\lambda}
\newcommand{\la}{\langle}
\newcommand{\ra}{\rangle}
\newcommand{\RR}{\mathbb{R}}
\newcommand{\Rbb}{\mathbb{R}}

\newcommand{\Null}{{\rm Null}}

\newcommand{\Sbb}{\mathbb{S}}
\newcommand{\aff}{{\rm aff}}

\newcommand{\lin}{{\rm lin}}
\newcommand{\app}{{\rm app}}
\newcommand{\appl}{{\rm appl}}
\newcommand{\disten}{{\rm dist}}
\newcommand{\dxx}{\Delta x}
\newcommand{\dXX}{\Delta X}
\newcommand{\dxi}{\Delta \xi}

\newcommand{\dZZ}{\Delta Z}
\newcommand{\dGG}{\Delta \Gamma}

\newcommand{\dtG}{\Delta\widetilde\Gamma}
\newcommand{\mA}{\mathcal{A}}
\newcommand{\mB}{\mathcal{B}}
\newcommand{\mC}{\mathcal{C}}

\newcommand{\mG}{\mathcal{G}}
\newcommand{\mH}{\mathcal{H}}
\newcommand{\mI}{\mathcal{I}}

\newcommand{\mN}{\mathcal{N}}
\newcommand{\mO}{\mathcal{O}}
\newcommand{\mQ}{\mathcal{Q}}
\newcommand{\mT}{\mathcal{T}}
\newcommand{\mU}{\mathcal{U}}
\newcommand{\mW}{\mathbb{W}}
\newcommand{\mX}{\mathbb{X}}
\newcommand{\mY}{\mathbb{Y}}
\newcommand{\mZ}{\mathbb{Z}}

\newcommand{\dist}{{\rm dist}}
\newcommand{\Diago}{{\rm Diag}}

\newcommand{\tIm}{{\rm Im}}
\newcommand{\nn}{{\nonumber}}

\begin{document}
	
	\title{A quadratically convergent semismooth Newton method for
		nonlinear semidefinite programming without generalized Jacobian regularity}
		\titlerunning{A quadratically convergent semismooth Newton method without generalized Jacobian regularity} 
	\author{Fuxiaoyue Feng \and
		Chao Ding \and Xudong Li
	}
	\institute{F.X.Y. Feng\at
		Institute of Applied Mathematics, Academy of Mathematics and Systems Science, Chinese Academy of Sciences, Beijing,  P.R. China,  School of Mathematical Sciences, University of Chinese Academy of Science, Beijing,  P.R. China. \\
		\email{fengfuxiaoyue@amss.ac.cn}%
		\and
		C. Ding \at
		Institute of Applied Mathematics, Academy of Mathematics and Systems Science, Chinese Academy of Sciences, Beijing,  P.R. China.   \\
		\email{dingchao@amss.ac.cn} 
		\and
		X.D. Li \at
		School of Data Science, Fudan University, Shanghai,  P.R. China.  \\
		\email{lixudong@fudan.edu.cn}
	}
 
	\date{This version: November 23, 2024}
	
	\maketitle
	
	\begin{abstract}
		We introduce a quadratically convergent semismooth Newton method for nonlinear semidefinite programming that eliminates the need for the {generalized Jacobian regularity}, a common yet stringent requirement in existing approaches. Our strategy involves identifying a single nonsingular element within the Bouligand generalized Jacobian, thus avoiding the standard requirement for nonsingularity across the entire generalized Jacobian set, which is often too restrictive for practical applications. The theoretical framework is supported by introducing the weak second order condition (W-SOC) and the weak strict Robinson constraint qualification (W-SRCQ). These conditions not only guarantee the existence of a nonsingular element in the generalized Jacobian but also forge a primal-dual connection in linearly constrained convex quadratic programming. The theoretical advancements further lay the foundation for the {\tvio algorithmic design} of a novel semismooth Newton method, which integrates a correction step to address degenerate issues. Particularly, this correction step ensures the local convergence as well as a superlinear/quadratic convergence rate of the proposed method. Preliminary numerical experiments corroborate our theoretical findings and underscore the practical effectiveness of our method.   
		\vskip 10 true pt
		
		\vskip 10 true pt
		
	\end{abstract}

\section{Introduction}\label{section:intro}

Let $\mX$ and $\mY$ be two Euclidean spaces each equipped with  {an inner} product $\langle \cdot,\cdot\rangle$ and its induced norm $\|\cdot\|$. Consider the following optimization problem (OP):
\begin{align}
	\begin{aligned}
		\min\quad        & f(x) \\
		\mbox{s.t.}\quad & G(x) \in K,
	\end{aligned}
	\label{prog:P}
\end{align}
where $K \subseteq \mY$ is a nonempty closed convex set and $f: \mX \to \mathbb{R}$, $G: \mX \to \mY$ are twice continuously differentiable. 

Let $L:{\mathbb X}\times{\mathbb Y}\to\mathbb{R}$ be the Lagrangian function of problem \eqref{prog:P} defined by
\begin{equation*} 
	L(x;y):=f(x)+ \langle y, {  G}(x) \rangle, \quad (x,y)\in{\mathbb X}\times{\mathbb Y}.
\end{equation*}
For any $y\in{\mathbb Y}$, denote the derivative of $L(\cdot;y)$ at $x\in\mathbb{X}$ by $L'_x(x;y)$ and denote the adjoint of $L'_x(x;y)$ by $\nabla_x L(x;y)$. The Karush-Kuhn-Tucker (KKT) optimality conditions for problem \eqref{prog:P} {\tvio take the following form}: 
\begin{align}
	\left\{
	\begin{aligned}
		&\nabla_x {L}(x,y) = 0, \\
		&y \in \mathcal{N}_{K}(G(x)),
	\end{aligned}
	\right. \label{KKT:P}
\end{align}
where  $\mathcal{N}_K(G(x))$ is the normal cone to $K$ at $G(x)$ in the context of convex analysis \cite{rockafellar1997convex}. {\tvio We say that} $(x,y) \in \mZ := \mX \times \mY$ is a KKT point of \eqref{prog:P} if it solves \eqref{KKT:P}. The set of Lagrange multipliers associated with a feasible point $x$ is defined by $M(x) = \{y \mid (x,y) \text{ is a KKT point}\}$. For a given feasible point $\overline{x}$, if $M(\overline{x}) \neq \emptyset$, then we call it a stationary point of \eqref{prog:P}. It is well-known that $(\overline{x}, \overline{y})$ is a KKT point of \eqref{prog:P} if and only if it is a solution of the nonsmooth system 
\begin{equation}\label{eq:natural-mapping}
	F(x,y) := \begin{bmatrix}
		\nabla_x L(x,y) \\
		-G(x) + \Pi_{K}(G(x) + y)
	\end{bmatrix}=0,
\end{equation}
where $\Pi_K : \mY \to \mY$ is the metric projection operator onto $K$, i.e., {\tvio$\Pi_K(z) = {\rm arg}\min_{y' \in K} \frac{1}{2}\| z - y'\|^2$} for $z \in \mY$. 

When solving $F(x,y) = 0$, one may encounter challenges related to the nonsmoothness of $F$, primarily arising from the nonsmooth nature of the metric projection operator $\Pi_{K}(\cdot)$, especially when $K$ is not a subspace. Fortunately, for various problems, including nonlinear programming (NLP), nonlinear second-order cone programming, and nonlinear semidefinite programming (NLSDP), the corresponding nonsmooth map $F$ exhibits the so-called (strong) semismoothness (see Definition \ref{def:semismooth}). This property is conducive to the use of nonsmooth Newton methods. Qi and Sun \cite{qi1993nonsmooth}  introduced a semismooth Newton method, which takes the following semismooth Newton step in each iteration:
$$\tvio (x^{k+1};y^{k+1}) = (x^k;y^k) - (U^k)^{-1}F(x^k,y^k)$$
with $U^k$ being selected from the Bouligand generalized Jacobian $\tvio \partial_B F(x^k,y^k)$ (see \eqref{eq:partialB} for definition) or the Clarke generalized Jacobian $\tvio \partial_C F(x^k,y^k)$ (see \eqref{eq:partialC} for definition). 
{The semismooth Newton method, renowned for its high efficiency, has found widespread applications across various domains. It has been effectively utilized for tackling variational inequalities and constrained optimization problems in function spaces \cite{ulbrich2011semismooth}, infinite-dimensional linear complementarity problems \cite{hintermuller2010semismooth} and computing the nearest correlation matrix \cite{qi2006quadratically}. Notably, {\tvio its efficiency has been} particularly pronounced in solving augmented Lagrangian subproblems, whose versatility is underscored by its implementation in several influential software packages, including SDPNAL \cite{zhao2010newton}, SDPNAL+ \cite{yang2015sdpnal+}, QSDPNAL \cite{li2018qsdpnal}, and SSNAL \cite{li2018highly}. For an in-depth exploration, one may refer to \cite{Qi1999,hintermuller2010semismooth,facchinei2003finite}.}

It is well-known (see e.g., \cite{qi1993nonsmooth,kummer1992newton,qi1993convergence}) that the (quadratic) superlinear convergence of semismooth Newton methods {\tvio hinges on the} (strong) {g-semismoothness (see Definition \ref{def:semismooth})} of $F$ and the uniform boundedness of  $\{(U^k)^{-1}\}$. {To guarantee such uniform boundedness, typically, one {\tvio has to assume that the KKT mapping $F$ satisfies a certain generalized Jacobian regularity condition at} the solution point $(\overline{x},\overline{y})$, for example, all elements in some generalized Jacobian sets of $F$ at $(\overline{x},\overline{y})$ are nonsingular.} Particularly, one may assume that every element in \( \partial_B F(\overline{x},\overline{y}) \) or even \( \partial_C F(\overline{x},\overline{y}) \) is nonsingular, known as the BD-regularity or CD-regularity, respectively. However, this condition is frequently demanding and may not always be fulfilled in practical applications. Thus, the following important question arises: 
\begin{quotation}
	{\em Can we design a locally superlinear or quadratic convergent semismooth Newton method for solving \eqref{eq:natural-mapping} without assuming the generalized Jacobian regularity of $F$?
    }
\end{quotation}
{\tvio In this paper, we provide an affirmative answer to this question in the context of nonlinear semidefinite programming (NLSDP).} Specifically, for the NLSDP, we design a semismooth Newton method with a novel correction step to avoid the requirement of nonsingularity across the entire generalized Jacobian of $F$, and prove its fast local convergence rates, thus advancing the traditional paradigms.{\tvio As illustrated in Example \ref{example:4}, despite the presence of singular elements in the Bouligand generalized Jacobian $\partial_B F(\overline{x},\overline{y})$ at the KKT point $(\overline{x},\overline{y})$ and the lack of the calmness property for $F$, our method demonstrates local quadratic convergence, as shown in Figure \ref{fig:ex4}.}

For the optimization problem \eqref{prog:P}, the characterization of the generalized Jacobian regularity for the nonsmooth function $F$ is closely intertwined with perturbation theory, a fundamental topic in optimization. Indeed, the perturbation theory is crucial for both  theoretical insights and practical algorithm design. Over the past four decades, significant advancements in this field have been well-documented in the literature \cite{bonnans2013perturbation,klatte_nonsmooth_2002,rockafellar2009variational,facchinei2003finite,mordukhovich2006variational}, 
especially concerning \eqref{prog:P} with $K$ being a polyhedral set (i.e., $K$ is the intersection of a finite number of half-spaces). 
However, less exploration has been done regarding non-polyhedral cases, such as $K$ being the set of positive semidefinite matrices.

In the realm of nonlinear semidefinite programming, Sun \cite{sun2006strong} established that under the Robinson constraint qualification (RCQ) (Definition \ref{def:RCQ}), the nonsingularity of $\partial_C F(\overline x,\overline y)$ (i.e., the CD-regularity of $F$) at a local optimal solution $\overline{x}$ and an associated  multiplier $\overline{y}$ is equivalent to the strong second-order sufficient condition (S-SOSC) (Definition \ref{def:S-SOSC}) and the constraint nondegeneracy (Definition \ref{def:constraint nondegeneracy}), and is also equivalent with the strong regularity introduced by Robinson \cite{robinson1980strongly} of the solution to the KKT system \eqref{KKT:P}. For linear SDP problems, Chan and Sun \cite{chan2008constraint} provided deeper insights into the strong regularity, demonstrating that under {\tvio the RCQ}, primal and dual constraint nondegeneracies, the nonsingularity of  $\partial_B F(\overline x,\overline y)$ (i.e., the BD-regularity of $F$), CD-regularity and the strong regularity at $(\overline x,\overline y)$ are all equivalent. This comprehensive understanding underscores that the conditions for the nonsingularity of all elements in the generalized Jacobians of $F$ at $(\overline x,\overline y)$  may be stringent and not consistently be satisfied in practice. For instance, the S-SOSC condition presupposes the {isolatedness} of a local optimal solution, a premise that is often violated in statistical optimization applications such as LASSO \cite{tibshirani1996regression}. Consequently, it is crucial to develop a new framework of semismooth Newton methods that can bypass the generalized Jacobian regularity and still achieve a local superlinear or quadratic convergence rate.

In this paper, we propose to relax the stringent generalized Jacobian regularity conditions by identifying a single nonsingular element within the Bouligand generalized Jacobian. However, characterizing the nonsingularity of even a single element within the generalized Jacobians of $F$ at a KKT point $(\overline{x},\overline{y})$, especially in non-polyhedral cases, remains an unresolved challenge. When $K$ is polyhedral, {\tvio Proposition 6 in Izmailov and Solodov \cite{izmailov2003karush}} indicated that {the weak strict Robinson constraint qualification (W-SRCQ) and the weak second-order condition (W-SOC) (see Definitions \ref{def:W-SRCQ} and \ref{def:W-SOC}, respectively) guarantee the existence of a nonsingular element in the parametric generalized Jacobian. Here the parametric generalized Jacobian, defined as  differentials of all possible parametric systems, contains $\partial_B F(\overline x,\overline y)$. Indeed, a straightforward deduction from their findings indicates that the W-SRCQ and W-SOC also guarantee the presence of  a nonsingular element in $\partial_B F(\overline x,\overline y)$.} However, these conditions are inadequate when applied to NLSDP. Specifically, the W-SRCQ and W-SOC are insufficient to ensure the existence of a nonsingular element in $\partial_B F(\overline x,\overline y)$, as demonstrated by our counterexample (Example \ref{example:non-SRCQ}). Meanwhile, in NLSDP, the presence of a nonsingular generalized Jacobian element, as in Example \ref{example:4}, does not necessarily imply calmness, an important property described in \cite[Definition 9(30)]{rockafellar2009variational}. The lack of calmness may lead to slow progress of algorithms for solving the KKT system \eqref{eq:natural-mapping}, as the associated solution mapping exhibits heightened sensitivity to small perturbations.

In this paper, we show that {the existence of a nonsingular element in} $\partial_B F(\overline x,\overline y)$ can be ensured under the W-SOC with constraint nondegeneracy or the S-SOSC with the W-SRCQ at a KKT point $(\overline{x},\overline{y})$. Additionally, our results indicate that the W-SRCQ and W-SOC are sufficient for the existence of a nonsingular element in $\partial_C F(\overline x,\overline y)$, and are also necessary for convex cases. More interestingly, we uncover a profound primal-dual connection between the W-SOC and W-SRCQ for linearly constrained convex quadratic semidefinite programming. These connections mirror those found between the strict Robinson constraint qualification (SRCQ) (Definition \ref{def:SRCQ}) and the second-order sufficient condition (SOSC) (Definition \ref{def:SOSC}) for linearly constrained quadratic semidefinite programming \cite{han2018linear}, as well as those between the constraint nondegeneracy and the S-SOSC in the setting of linear semidefinite programming \cite{chan2008constraint}. We summarize {our} results and connections in Figure \ref{figure:con}. Based on the theoretical findings, we are able to design an inexact semismooth Newton method with a correction step. Each iteration commences with a semismooth Newton step, followed by a correction step designed to ensure the nonsingularity. Under the W-SOC and constraint nondegeneracy (or the S-SOSC and W-SRCQ), we show the local convergence and superliner convergence of the designed algorithm.  Furthermore, if the function is twice Lipschitz continuously differentiable, the convergence rate can be further improved to quadratic.
\begin{figure}[h]
\centering
\includegraphics[scale=0.6]{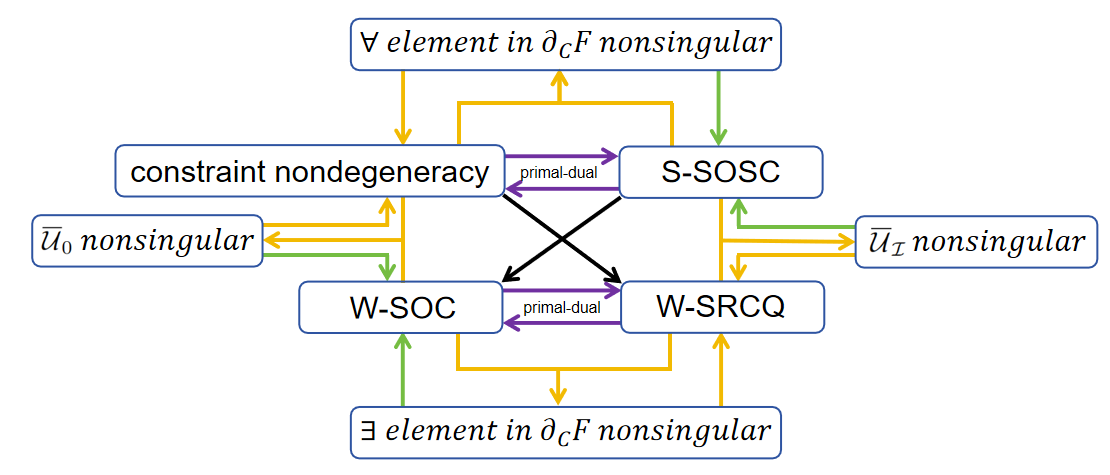}
\caption{Connections of second order conditions, constraint qualifications and nonsingularity of generalized Jacobians. Here the {\color{black}black} arrow holds for \eqref{prog:P}; the {\color{orange}orange} arrow holds for NLSDP; the {\color{green}green} arrow holds for CSDP; and the {\color{violet} violet} arrow holds for QSDP. Moreover, ``primal-dual'' means that the starting point of the arrow for the primal problem implies the arrival point of the arrow for the dual problem.}
\label{figure:con}
\end{figure}
 
The remaining parts of this paper are organized as follows. Section \ref{section:pre} provides some basic definitions and preliminary results concerning the variational properties of the metric projection projector over the positive semidefinite cone. Section \ref{section:nonsingularity} delves into the {\tvio sufficient conditions and necessary conditions} for the existence of one nonsingular element in the generalized Jacobians of $F$ for NLSDP. The primal-dual connections of the W-SOC and the W-SRCQ are explored in Section \ref{section:cha} for convex {quadratic} semidefinite programming. Then, we present a semismooth Newton method with a correction step in Section \ref{section:5}, and some preliminary numerical results in Section \ref{section:numerical}. Finally, we conclude this paper and highlight potential directions for future research in Section \ref{section:conclusion}.

\section{Preliminaries}
\label{section:pre}

{We first give the following notation that will be used throughout the paper. Let $\mathbb{X}$ and $\mathbb{Y}$ be two real Euclidean spaces. Given a set $S\subseteq \mathbb{X}$ and a point $x\in \mathbb{X}$, the distance from $x$ to $S$ is denoted by
	$${\rm dist}(x,S):=\mbox{inf}\{\|y-x\|\,|\ y\in S\}.$$ For a linear operator ${\cal A}: \mathbb{X}\rightarrow \mathbb{Y}$, ${\cal A}^*$ denotes the adjoint of the  linear operator ${\cal A}$.  We denote by $A^T$ the transpose of a given matrix $A$. For a mapping
	$\Phi:\mathbb{X}\rightarrow \mathbb{Y}$ and  $x\in \mathbb{X}$, $\Phi'(x)$
	stands for classical derivative or the Jacobian of $\Phi$ at $x$ and $\nabla \Phi(x)$ the adjoint of the Jacobian. 
We use  $\mathbb{S}^n$ to denote the linear space of all $n\times n$ real symmetric matrices and use $\mathbb{S}^n_+$ to denote the closed convex cone of all $n\times n$ positive semidefinite matrices in $\mathbb{S}^n$.  For a function $g:\mathbb{X} \rightarrow \mathbb{R}$, we denote $g^+(x):=\max \{0, g(x)\}$ and if it is vector-valued then the maximum is taken componentwise. 
} Some other notations are listed below:
 \begin{itemize}
 	\item Let ${\cal O}^{n}$ be the set of all $n\times n$ orthogonal matrices.
 	\item For any $Z\in\mathbb{R}^{m\times n}$, we denote by $Z_{ij}$ the $(i,j)$-th entry of $Z$.
 	\item For any $Z\in\mathbb{R}^{m\times n}$ and a given index set ${\cal J}\subseteq  \{1,\ldots, n\}$,  we use $ Z_{{\cal J}}$ to denote the sub-matrix of $Z$ obtained by removing all the columns of $Z$ not in ${\cal J}$. In particular, we use $Z_{j}$ to represent the $j$-th column of $Z$, $j=1,\ldots,n$.
 	\item   Let ${\cal I}\subseteq \{1,\ldots, m\}$ and  ${\cal J}\subseteq \{1,\ldots, n\}$ be two index sets. For  any    $Z\in\mathbb{R}^{m\times n}$, we use $Z_{{\cal I}{\cal J}}$ to
 	denote the $|{\cal I}|\times|{\cal J}|$ sub-matrix of $Z$ obtained  by removing all the rows of $Z$ not in ${\cal I}$ and all the columns of $Z$ not in  ${\cal J}$.
 	\item We use $``\circ"$ to denote the {\tvio  Hadamard product} between matrices, i.e., for any two matrices $A$ and $B$ in $\mathbb{R}^{m\times n}$ the $(i,j)$-th entry of $  Z:= A\circ B \in \mathbb{R}^{m\times n}$ is
 	$Z_{ij}=A_{ij} B_{ij}$.
 	\item Let ${\rm Diag}(\cdot):\mathbb{R}^n\to \mathbb{S}^n$ be the linear map such that for any $x\in \mathbb{R}^n$, ${\rm Diag}(x)$ denotes the diagonal matrix whose $i$-th diagonal entry is $x_i$, $i=1, \ldots, n$.
 	\item Let $\lin(C)$ be the linearity space of a closed convex set $C\subseteq \mathbb{X}$, i.e., the largest subspace in $C$.  Meanwhile, let $\aff(C)$ be the affine hull of $C$, i.e., the smallest subspace containing $C$.
 	\item For a map $\Phi:\mX\to \mY$, define $\tIm(\Phi):=\{\Phi(x)\ |\ x\in\mX\}$ and $\Null(\Phi):=\{x\in\mX\ |\ \Phi(x)=0\}$.  
 	\item The {\tvio support function} of a set $S\subseteq \mX$ is defined by $\sigma(x,S)=\sup_{s\in S} \la x,s\ra $ for $x\in\mX$.
 	\item For $r>0$ and $\overline X\in\mX$, the closed ball centered at $\overline X$ with radius {\tvio$r$ is  denoted} by $\mathbb{B}({\overline{X}},r)=\{X\in\mX\ |\ \|X-\overline X\|\le r\}$. 
 	\item For $f,g:\mX\to\RR$, $f(x)=\Theta(g(x))$ for $x\to \overline x$ means that $f(x)=O(g(x))$ and $g(x)=O(f(x))$.
 \end{itemize}

	For a given map $\Phi:\mathbb{X}\rightarrow \mathbb{Y}$, the directional differential of {\tvio $\Phi$ at a point} $x\in \mX$ along a direction $d\in \mX$ is defined by $$\Phi'(x;d)=\lim_{t\downarrow 0}\frac{F(x+td)-F(x)}{t}\quad  \mbox{\tvio if this limit exists.}$$ 
	$\Phi$ is directionally differentiable at $x\in\mX$ if $\Phi'(x;d)$ exists for all $d\in\mX$. 

    Suppose that $\Phi: N\subseteq \mathbb{X} \to  \mathbb{Y}$ is a locally Lipschitz continuous function on the open set $N$. Then, according to Rademacher's theorem \cite{Rademacher1919}, $\Phi$ is almost everywhere  F-differentiable in $N$.  Let $D_\Phi$ be the set of points in $N$ where $\Phi$ is differentiable. Then the Bouligand generalized Jacobian of $\Phi$ at $x\in N$ is denoted by \cite{qi1993convergence}:
	\begin{equation}\label{eq:partialB}
		\partial_{B}\Phi(x):=\big\{ \lim_{D_{\Phi}\ni x^{k}\to x} \Phi'(x^{k}) \big\}\, ,
	\end{equation}
	and the Clarke generalized Jacobian  of $\Phi$ at $x\in N$ \cite{Clarke1983OptimizationNA}  takes the form:
	\begin{align}\label{eq:partialC}
		\partial_C \Phi(x)={\rm conv}\big\{\partial_{B} \Phi(x)\big\},
	\end{align} 
	where ``${\rm conv}$'' stands for the convex hull in the usual sense of convex analysis \cite{rockafellar1997convex}.
	The following g-semismoothness of a locally Lipschitz continuous function is an extension of the semismoothness. {\tvio The concept of semismoothness was initially introduced by Mifflin \cite{mifflin1977semismooth} for real-valued functions and later extended to vector-valued functions by Qi and Sun \cite{qi1993nonsmooth}. This concept was further extended to g-semismoothness by Gowda \cite{doi:10.1080/10556780410001697668}, which applies to functions without requiring directional differentiability.}
	\begin{definition}\label{def:semismooth}
	Let $\Phi: N\subseteq \mathbb{X} \to  \mathbb{Y}$ be a locally Lipschitz continuous function on the open set $N$. The function $\Phi$ is said to be  g-semismooth  at a point $x \in N$  if
	for any $y\to x $ and $V \in \partial_C \Phi(y)$,
	\[
	\Phi(y) -\Phi(x) - V(y-x) = o(||y- x|| )\, .
	\]
	The function $\Phi$ is said to be  strongly g-semismooth at $x$  if
	for any $y\to x $ and $V \in \partial_C \Phi(y)$,
	\[
	\Phi(y) - \Phi(x) - V(y- x) = O(||y-x||^2)\, .
	\]
	Furthermore, the  function $\Phi$ is said to be  (strongly) semismooth at $x\in N$  if (i) $\Phi$ is directional differentiable at $x$;
    and (ii) $\Phi$ is   (strongly) g-semismooth.
\end{definition}

\subsection{{Background in variational analysis}}

{In this subsection, we introduce some variational properties related to problem \eqref{prog:P}, which will be useful in our subsequent discussions.}

{For the closed convex set $K\subseteq \mathbb{Y}$, the tangent cone to $K$ at $y\in\mY$ is defined by}
\begin{equation}\label{eq:def-Tangent-g}
		\mT_K(y)=\{d\in\mY\ |\ \exists\, t^k\downarrow 0,\ {\dist}(y+t^kd,K)=o(t^k)\}.
\end{equation}
Let $\overline{x}$ be a given feasible solution to \eqref{prog:P}. It is well-known \cite[Theorem 3.9 and Proposition 3.17]{bonnans2013perturbation} that for the feasible solution $\overline{x}$, the following Robinson constraint qualification is equivalent to the nonemptyness and boundedness of the corresponding Lagrange multiplier set $M(\overline{x})$.
\begin{definition}\label{def:RCQ}
	The Robinson constraint qualification (RCQ) is said to hold at a feasible point $\overline x$ if
\begin{equation*}
G'(\overline x)\mX+\mT_K(G(\overline x))=\mY.
\end{equation*}
\end{definition}

{For a stationary solution $\overline{x}$ of \eqref{prog:P}, the strict Robinson constraint qualification condition is defined as follows.}
\begin{definition}\label{def:SRCQ}
Let $\overline{x}\in\mathbb{X}$ be a stationary point of \eqref{prog:P} with a Lagrange multiplier $\overline{y}\in M(\overline x)$. The strict Robinson constraint qualification (SRCQ) is said to hold at $\overline x$ with respect to $\overline y$ if
\begin{equation*}
G'(\overline x)\mX+\mT_K(G(\overline x))\cap \overline y^{\perp}=\mY,
\end{equation*}
{where $y^{\perp}:=\{s\in\mY\ |\ \la s,y\ra=0\}$ for any vector $y\in\mY$.}
\end{definition}
{It follows from \cite[Proposition 4.50]{bonnans2013perturbation} that the set of Lagrange multipliers $M(\overline x)$ is a singleton if the SRCQ hold.}
	
For the feasible point $x$, the critical cone ${\cal C}(x)$ of \eqref{prog:P} at $x$ is defined by
\begin{equation}\label{eq:critical-1}
\mC(x)=\{\dxx\in\mX\ |\ G'(x)\dxx\in \mT_K(G(x)),\ f'(x)\dxx=0\}.
\end{equation} 
If $\overline{x}$ is a stationary point with $\overline{y}\in M(\overline{x})$, then the critical cone ${\cal C}(\overline{x})$ takes the following form:
\begin{equation}\label{eq:critical-2}
\mC(\overline x)=\{\dxx\in\mX\ |\ G'(\overline x)\dxx\in\mT_K(G(\overline x))\cap \overline y^{\perp}\}.
\end{equation} 
The following definition on the constraint nondegeneracy of \eqref{prog:P}, introduced by Robinson \cite{robinson1980strongly}, is stronger than the SRCQ since $\lin(\mT_K(G(\overline x)))\subseteq \mT_K(G(\overline x))\cap \overline y^{\perp}$ (cf. \cite[Proposition 4.73]{bonnans2013perturbation}).
\begin{definition}\label{def:constraint nondegeneracy}
	The constraint nondegeneracy of \eqref{prog:P}  is said to hold at $\overline x$ if
	\begin{equation*}
		G'(\overline x)\mX+\lin(\mT_K(G(\overline x)))=\mY.
	\end{equation*}
\end{definition}

Recall that the inner and outer second order tangent sets (\cite[(3.49) and (3.50)]{bonnans2013perturbation}) to the given closed set $K\in\mathbb{Y}$ in the direction $d\in\mathbb{Y}$ {\tvio are defined}, respectively, by
\begin{equation*}
\mT^{i,2}_K(y,d):=\{w\in \mY\mid \disten(y+td+\frac{1}{2}t^2w,C)=o(t^2),\ t\ge 0\}
\end{equation*}
and
\begin{equation}\label{eq:def-second-order-tangent}
\mT^{2}_K(y,d):=\{w\in \mY\mid \exists\,t^k\downarrow 0,\ \disten(y+t^kd+\frac{1}{2}(t^k)^2w,C)=o((t^k)^2)\}.\nn
\end{equation}
Note that in general, $\mT^{i,2}_K(y,d)\neq\mT^{2}_K(y,d)$ even for {\tvio convex sets $K$} (cf. \cite[Section 3.3]{bonnans2013perturbation}).  However, it follows from \cite[Proposition 3.136]{bonnans2013perturbation} that if $K$ is a $C^2$-cone reducible  \cite[Definition 3.135]{bonnans2013perturbation} convex set (e.g., the polyhedron, second-order cone, positive semidefinite matrix cone and their Cartesian products), {\tvio then equality always holds}. In this case, ${\cal T}^{2}_{K}(y,d)$ will be simply called the second order tangent set to $K$ at $y\in \mathbb{Y}$ in the direction $d\in\mathbb{Y}$.

Suppose that $K$ in \eqref{prog:P} is $C^2$-cone reducible. The following second order sufficient condition is adopted from \cite[(3.276)]{bonnans2013perturbation}.
\begin{definition}\label{def:SOSC}
	Let $\overline{x}\in\mathbb{X}$ be a stationary point of \eqref{prog:P} with a Lagrange multiplier $\overline{y}\in M(\overline x)$. The second order sufficient condition (SOSC) is said to hold at $(\overline x,\overline y)$ if
\begin{align}\label{eq:SOSC-P}
\la \dxx,\nabla^2_{xx} L(\overline x,\overline y)\dxx\ra-\sigma(\overline y,\mathcal{T}^2_{K}(G(\overline x),G'(\overline x)\dxx))>0 \quad \forall\, \dxx\in\mC(\overline x)\backslash\{0\}.
\end{align}
\end{definition}

The following definition of {\tvio  the strong second order sufficient condition} for \eqref{prog:P} is an extension of  \cite[Definition 3.2]{sun2006strong} for {\tvio the NLSDP}.  It is worth noting that when $M(\overline{x})$ is a singleton, both of them are the same.
\begin{definition}\label{def:S-SOSC}
	Let $\overline{x}\in\mathbb{X}$ be a stationary point of \eqref{prog:P} with a Lagrange multiplier $\overline{y}\in M(\overline x)$. The strong second order sufficient condition (S-SOSC) is said to hold at $(\overline x,\overline y)$ if
\begin{equation*}
\la \dxx,\nabla^2_{xx} L(\overline x,\overline y)\dxx\ra-\sigma(\overline y,\mathcal{T}^2_{K}(G(\overline x),G'(\overline x)\dxx))>0\quad \forall\, \dxx\in\app(\overline x,\overline y)\backslash\{0\}, 
\end{equation*}
where $\app(\overline x,\overline y)$ is the {\tvio outer approximation} of the affine hull of the critical cone ${\cal C}(\overline{x})$ of \eqref{prog:P} with respect to $(\overline{x},\overline{y})$, i.e.,
\begin{equation}\label{eq:def-app}
	\app(\overline x,\overline y):=\{\dxx\in\mathbb{X}\mid G'(\overline x)\dxx\in\aff(\mT_K(G(\overline x))\cap \overline y^{\perp})\}.
\end{equation}
\end{definition}

\subsection{Eigenvalue decomposition of symmetric matrices}

In this subsection, we introduce some useful preliminary results on eigenvalue decompositions of real symmetric matrices and the differentiability of the metric projectors over the positive semidefinite matrix cone $\mathbb{S}^n_+$.

Let $A \in\mathbb{S}^n$ be given. We use $\lambda_{1}(A)\ge \lambda_2(A) \ge \ldots \ge \lambda_{n}(A)$ to denote the eigenvalues of $A$ (all real and counting multiplicity) arranging in  nonincreasing order and use $\lambda(A)$ to denote the vector of the ordered eigenvalues of $A$. Let $\Lambda(A):= {\rm Diag}(\lambda(A))$. Consider the eigenvalue decomposition of $A$, i.e., $A={P} \Lambda(A){P}^{T}$, where ${P}\in{\cal O}^{n}$. By considering the index sets of positive, zero, and negative eigenvalues of $A$, we are able to write $A$ in the following form
\begin{equation}\label{eq:eig-de}
	A= \left[\begin{array}{ccc}
		{P}_{\alpha} & {P}_{\beta} & {P}_{\gamma}
	\end{array}\right] \left[\begin{array}{ccc}
		\Lambda(A)_{\alpha\alpha} & 0 & 0 \\ [3pt]
		0 & 0 & 0\\ [3pt]
		0 & 0 & \Lambda(A)_{\gamma\gamma}
	\end{array}\right]\left[\begin{array}{c}
		{P}_{\alpha}^T \\ [3pt]
		{P}_{\beta}^T \\
		{P}_{\gamma}^T
	\end{array}\right],
\end{equation}
where $\alpha$, $\beta$ and $\gamma$ are index sets defined by
\begin{equation}\label{eq:index}
	\alpha:=\{i \mid \lam_i>0\},\quad\beta:=\{i\ |\ \lam_i=0\}\quad\text{and}\quad\gamma:=\{i\ |\ \lam_i<0\}.
\end{equation}
We use ${\cal O}^{n}(A)$ to denote the set of all orthogonal matrices $P\in{\cal O}^{n}$ satisfying \eqref{eq:eig-de}.

The next proposition, which was stated in \cite[Proposition 4.4]{sun2002strong}, illustrates that the eigenvectors of symmetric matrices, though not continuous, are upper Lipschitz continuous. 
\begin{proposition}\label{prop:dis-P}
	Given $A\in\Sbb^n$, for any $H\in\Sbb^n$, let $P$ be an orthogonal matrix such that $X+H=P\Lambda(A+H)P^T$. Then, as $H\to 0$, we have
	\begin{align}\nn
		\disten(P,\mO^n(A))=O(\| H\| ).
	\end{align}
\end{proposition}

Let $\Pi_{\mathbb{S}^n_+}:\mathbb{S}^n\to\mathbb{S}^n$ be the metric projection operator over the positive semidefinite matrix cone $\mathbb{S}^n_+$, i.e.,
\begin{equation}\label{eq:Pi-SDP}
\Pi_{\mathbb{S}^n_+}(A):={\rm arg}\min_{X\in\mathbb{S}^n_+}\left\{\frac{1}{2}\|X-A\|^2\right\},\quad A\in\mathbb{S}^n.
\end{equation}
It is clear that $\Pi_{\mathbb{S}^n_+}(\cdot)$ is globally Lipschtiz continuous. Furthermore, it follows from \cite[Theorem 4.7]{sun2002semismooth} (see also \cite[Proposition 9]{pang2003semismooth}) that $\Pi_{\mathbb{S}^n_+}(\cdot)$ is directionally differentiable at any $A\in\mathbb{S}^n$ and the directional derivative at $A$ along {\tvio the direction} $H\in\mathbb{S}^n$ is given by 
\begin{equation} \label{eq:dd-projection}
	\Pi_{\mathbb{S}_+^n} ^{\prime}(A;H) =  {P} \left[
	\begin{array}{ccc}
		{\widetilde H}_ {\alpha\alpha}   & {\widetilde H}_{\alpha\beta} &\Xi_{\alpha\gamma} \circ {\widetilde H}_{\alpha\gamma}\\
		\widetilde{H}_{\alpha\beta}^T &  \Pi_{{\cal S}_+^{\rvert\beta\lvert} }({\widetilde H}_{\beta\beta}) &0\\
		\Xi_{\alpha\gamma}^T \circ {\widetilde H}_{\alpha\gamma}^T & 0 & 0
	\end{array}
	\right]  {P}^T,
\end{equation}
with $P\in\mathcal{O}^n(A)$, ${\widetilde H}:={P}^T H{P}$  and $\Xi\in\mathbb{S}^n$ is defined by 
\begin{equation}\label{eq:Xi}
	\Xi_{ij}:=\frac{\max\{\lambda_{i}(A),0\}-\max\{\lambda_{j}(A),0\}}{\lambda_{i}(A)-\lambda_{j}(A)},\quad i,j=1,\ldots,n,
\end{equation} where $0/0$ is defined to be $1$. Meanwhile, it is well-known \cite[Proposition 4.5 and Theorem 4.13]{sun2002semismooth} that the metric projection operator $\Pi_{\mathbb{S}^n_+}(\cdot)$ is strongly semismooth everywhere.

The following result on the characterizations of the Bouligand generalized Jacobian and Clarke generalized Jacobian  of $\Pi_{\mathbb{S}_+^n}(\cdot)$ is taken from \cite[Lemma 11]{pang2003semismooth}.

\begin{proposition}\label{prop:dif-B-F}
	Given $A\in\Sbb^n$, {\tvio it holds that} ${\cal V}\in\partial_B\Pi_{\Sbb^n_+}(A)$ ($\partial_C\Pi_{\Sbb^n_+}(A)$) if and only if there is ${\cal W}\in\partial_B\Pi_{\Sbb^{|\beta|}_+}(0)$ ($\partial_C\Pi_{\Sbb^{|\beta|}_+}(0)$) such that for any $H\in \Sbb^n$,
\[
{\cal V}(H)=P
\begin{bmatrix}
\widetilde{H}_{\alpha\alpha}&\widetilde{H}_{\alpha\beta}&\Xi_{\alpha\gamma}\circ\widetilde{H}_{\alpha\gamma} \\
\widetilde{H}_{\beta\alpha}&{\cal W}(\widetilde{H}_{\beta\beta})&0 \\
\Xi_{\gamma\alpha}\circ\widetilde{H}_{\gamma\alpha}&0&0 
\end{bmatrix}
P^T,
\]
where $P\in{\cal O}^n(A)$ and $\widetilde{H}=P^THP$. In particular, since both the zero mapping ${\cal W}\equiv 0$  and the identity mapping ${\cal W}\equiv {\cal I}$ from $\mathbb{S}^{|\beta|}\to\mathbb{S}^{|\beta|}$ are elements of $\partial_B\Pi_{\Sbb^{|\beta|}_+}(0)$, the following two mappings ${\cal V}_0$ and ${\cal V}_{\mI}$ defined by
\begin{equation}\label{eq:exactV0}
	{\cal V}_0(H)=P
	\begin{bmatrix}
		\widetilde{H}_{\alpha\alpha}&\widetilde{H}_{\alpha\beta}&\Xi_{\alpha\gamma}\circ\widetilde{H}_{\alpha\gamma} \\
		\widetilde{H}_{\beta\alpha}&0&0 \\
		\Xi_{\gamma\alpha}\circ\widetilde{H}_{\gamma\alpha}&0&0 
	\end{bmatrix}
	P^T,\quad H\in \Sbb^n
\end{equation}
	and
	\begin{equation}\label{eq:exactV1}
			{\cal V}_{\mI}(H)=P
		\begin{bmatrix}
			\widetilde{H}_{\alpha\alpha}&\widetilde{H}_{\alpha\beta}&\Xi_{\alpha\gamma}\circ\widetilde{H}_{\alpha\gamma} \\
			\widetilde{H}_{\beta\alpha}&\widetilde{H}_{\beta\beta}&0 \\
			\Xi_{\gamma\alpha}\circ\widetilde{H}_{\gamma\alpha}&0&0 
		\end{bmatrix}
		P^T,\quad  H\in \Sbb^n 
	\end{equation}
	 are elements of $\partial_B\Pi_{\Sbb^n_+}(A)$. {Moreover, ${\cal V}_{0}$ and ${\cal V}_{\mI}$ are {\tvio independent of} the choice of $P\in \mathcal{O}^n(A)$.}
\end{proposition}

\section{Existence of nonsingular elements of the generalized Jacobians}\label{section:nonsingularity}

As mentioned in the introduction, our preliminary objective is to explore the existence of nonsingular elements within the generalized Jacobians of $F$ defined in \eqref{eq:natural-mapping}. This investigation sets the stage for the subsequent development of a semismooth Newton method that achieves quadratic convergence without the BD-regularity. To commence, we introduce the concept of the weak strict Robinson constraint qualification (W-SRCQ), which serves as an extension to the SRCQ condition (Definition \ref{def:SRCQ}).
	
\begin{definition}\label{def:W-SRCQ}
Let $\overline{x}\in\mathbb{X}$ be a stationary point of \eqref{prog:P} with a Lagrange multiplier $\overline{y}\in M(\overline{x})$. The weak strict Robinson constraint qualification (W-SRCQ) is said to hold at $\overline x$ with respect to $\overline{y}$ if
\[
G'(\overline x)\mX+\aff(\mT_K(G(\overline x))\cap \overline{y}^{\perp})=\mY.
\]
\end{definition}

\begin{remark}
	The W-SRCQ condition is referred to as the {\it weak linear independence constraint qualification} (W-LICQ) in the context of nonlinear variational inequalities, as explored in \cite{izmailov2003karush}. Nevertheless, we choose not to adopt the W-LICQ terminology in this context. This is because the W-SRCQ is not simply a weaker form of classical LICQ; rather, it is a distinct concept that arises by incorporating the affine hull into the SRCQ.
\end{remark}

As can be inferred directly from the definitions (Definitions \ref{def:SRCQ} and \ref{def:W-SRCQ}), the W-SRCQ  is weaker than the SRCQ. This implies that the W-SRCQ is also milder than the constraint non-degeneracy condition as outlined in Definition \ref{def:constraint nondegeneracy}. Moreover, the example provided below demonstrates that W-SRCQ is, in fact, even weaker than the RCQ defined by Definition \ref{def:RCQ}.
\begin{example}\label{example:1}
	Consider the following quadratic SDP problem:
	\begin{align}\label{eq:prob_example1}
		\begin{aligned}
			\min\quad        &\frac{1}{2}\| X_{11}\| ^2-\frac{1}{2}\| X_{22}\| ^2	 \\
			\mbox{s.t.}\quad & X_{12}=0, \quad X_{22}=0, \\
			& X\in \Sbb^n_+ , 
		\end{aligned}
	\end{align}
	where $X=\begin{bmatrix} X_{11} &X_{12}\\X_{21}&X_{22}\end{bmatrix}\in \mathbb{S}^n$ with  $X_{11}\in\RR^{l_1\times l_1}$ and $X_{22}\in\RR^{l_2\times l_2}$. It is easy to verify that $\overline{X}=0$ is a local optimal solution with the following Lagrange multiplier set 
	$$ M(\overline{X})=\left\{(\xi_{12},\xi_{22},\Gamma)\in\mathbb{R}^{l_1\times l_2}\times\mathbb{R}^{l_2\times l_2}\times \mathbb{S}^n_+ \mid \begin{bmatrix} 0 &\xi_{12}\\ \xi_{12}^T&\xi_{22} \end{bmatrix}+\Gamma=0\right\}.
	$$ 
	Thus, $(\overline{\xi}_{12},\overline{\xi}_{22},\overline{\Gamma})=(0,0,0)\in M(\overline{X})$. Consider the KKT point $(\overline{X},\overline \xi_{12},\overline\xi_{22},\overline{\Gamma})=(0,0,0,0)$. We know that the index sets defined by \eqref{eq:index} are $\alpha=\gamma=\emptyset$ and $\beta=\{1,\ldots,n\}$. It then can be checked directly by Definition \ref{def:W-SRCQ} that the W-SRCQ holds at $\overline{X}$ with respect to $(\overline{\xi}_{12},\overline{\xi}_{22},\overline{\Gamma})$. Meanwhile, since the multiplier set $M(\overline{X})$ is unbounded, we know that the RCQ does not hold at $\overline{X}$.
\end{example} 

Next, we {\tvio are to introduce the concept} of the weak second order condition for a KKT solution $(\overline{x},\overline{y})$ of \eqref{prog:P}, which is a generalization of the SOSC (Definition \ref{def:SOSC}).
\begin{definition}\label{def:W-SOC}
	Let $\overline{x}\in\mathbb{X}$ be a stationary point of \eqref{prog:P} with a Lagrange multiplier $\overline{y}\in M(\overline{x})$. The weak second order condition (W-SOC) is said to hold at $(\overline{x},\overline{y})$ if
\begin{equation}\label{eq:W-SOC-P}
\la \dxx,\nabla^2_{xx} L(\overline x,\overline y)\dxx\ra-\sigma(\overline y,\mathcal{T}^2_{K}(G(\overline x),G'(\overline x)\dxx))>0 \quad \forall\, \dxx\in\appl(\overline x,\overline y)\backslash\{0\},
\end{equation}
where $\appl(\overline x,\overline y)$ is defined by
\begin{equation}\label{eq:def-appl}
\appl(\overline x,\overline y):=\left\{\dxx\in\mathbb{X}\mid  G'(\overline x)\dxx\in\lin(\mT_K(G(\overline x))\cap \overline y^{\perp})\right\}.
\end{equation}
\end{definition}

\begin{remark}
	The W-SOC is referred to as the {\em second order condition} for nonlinear variational inequalities in \cite{izmailov2003karush}. However, to emphasize Definition \ref{def:W-SOC} is a weaker variant of the SOSC, we prefer the term of W-SOC here. Moreover, it can be seen from the definitions that at a KKT point, the following implications hold:
		\[
		\mbox{S-SOSC}\quad \Longrightarrow \quad \mbox{SOSC}\quad \Longrightarrow \quad \mbox{W-SOC}.
		\]
	Furthermore, it is important to note that the W-SOC does not suffice to ensure local optimality for general optimization problems. Additionally, within the domain of convex problems, the W-SOC alone does not necessarily lead to the isolation of the optimal solution.
\end{remark}

Henceforth, our primary focus will be on the following nonlinear semidefinite programming (NLSDP) problem:
\begin{equation}\label{prog:NLSDP}
	\begin{array}{cl}
		\min & f(x) \\ [3pt]
		\mbox{s.t.} & h(x)=0, \\ [3pt]
		& g(x)\in \Sbb^n_+.
	\end{array}
\end{equation}
It should be noted that the results derived herein can be similarly extended through straightforward generalizations to cases where $K$ in \eqref{prog:P} represents the Cartesian product of a finite number of positive semidefinite cones and zero vectors. Furthermore, considering that $\mathbb{R}^n$ can be represented as the Cartesian product of $n$ one-dimensional positive semidefinite cones $\mathbb{S}^1$, our analysis encompasses {\tvio classical nonlinear programming (NLP)} as well. In addition, we call the NLSDP \eqref{prog:NLSDP} {\tvio a convex semidefinite programming (CSDP) problem} if $f$ is convex, $h$ is affine and $g$ is $\mathbb{S}^n_+$-convex, i.e., for any $x,y\in\mX$ and $t\in(0,1)$,
 $$g(tx + (1 - t)y) - t g(x) - (1 - t)g(y)\in\Sbb^n_+.$$

For the NLSDP \eqref{prog:NLSDP}, the nonsmooth system \eqref{eq:natural-mapping} which is  equivalent to the KKT optimality condition \eqref{KKT:P}, takes the following form:
\begin{equation}\label{eq:SDP-natural-mapping}
F(x,\xi,\Gamma)=\begin{bmatrix}
&\nabla_xL(x,\xi,\Gamma) \\
& h(x) \\
&-g(x)+\Pi_{\Sbb^n_+}(g(x)+\Gamma) 
\end{bmatrix} = 0,
\end{equation}
where the associated Lagrange function is given by $L(x,\xi,\Gamma)=f(x)+\la \xi,h(x)\ra+\la \Gamma,g(x)\ra$ and $\Pi_{\mathbb{S}^n_+}(\cdot)$ is the metric projection operator over $\mathbb{S}^n_+$ defined by \eqref{eq:Pi-SDP}.

Let $\overline x$ be a stationery solution of \eqref{eq:SDP-natural-mapping} with the Lagrange multipliers $(\overline{\xi},\overline{\Gamma})\in M(\overline{x})$. Denote $\overline A=g(\overline{x})+\overline{\Gamma}$. Consider the corresponding eigenvalue decomposition \eqref{eq:eig-de} of $\overline A$ with the index sets $\alpha$, $\beta$ and $\gamma$ defined by \eqref{eq:index}, i.e.,
\begin{equation}\label{eq:eig-decomp-SDP}
A=g(\overline x)+\overline\Gamma =\overline{P}
\left[\begin{array}{ccc}
\overline\Lambda_{\alpha\alpha} & 0 & 0 \\ [3pt]
0 & 0 & 0 \\  [3pt]
0 & 0 & \overline\Lambda_{\gamma\gamma}
\end{array}\right]
\overline{P}^T,
\end{equation}
where $\overline\Lambda=\Lambda(A)=\Diago(\overline\lambda_i)$ with $\overline{\lambda}:=\lambda(A)$ and $\overline{P}\in \mO^n(A)$. Thus, the tangent cone ${\cal T}_{\mathbb{S}^n_+}(g(\overline{x}))$ of $\mathbb{S}^n_+$ defined by \eqref{eq:def-Tangent-g} has the following explicit expression (cf. e.g., \cite[(17)]{sun2006strong}):
\begin{equation}\label{eq:tangent-SDP}
\mT_{\Sbb^n_+}(g(\overline x))=\left\{\overline{P}B\overline{P}^T\in\Sbb^n\mid\begin{bmatrix} B_{\beta\beta}& B_{\beta\gamma}\\B_{\gamma\beta}&B_{\gamma\gamma}\end{bmatrix}\succeq 0\right\}. 
\end{equation}
{\tvio Furthermore, the detailed characterizations of the following sets can be found in \cite{sun2006strong}.}
\begin{eqnarray}
\mT_{\Sbb^n_+}(g(\overline x))\cap \overline\Gamma^{\perp}&=&\left\{\overline{P}B\overline{P}^T\in\Sbb^n\mid B_{\beta\beta}\succeq 0,\ B_{\beta\gamma}=0,\ B_{\gamma\gamma}=0\right\}, \label{eq:critical-cone-SDP} \\ [3pt] 
\lin(\mT_{\Sbb^n_+}(g(\overline x)))&=&\left\{\overline{P}B\overline{P}^T\in\Sbb^n\mid B_{\beta\beta}= 0,\ B_{\beta\gamma}=0,\ B_{\gamma\gamma}=0\right\}, \label{eq:lin-tangent-SDP} \\ [3pt]
\aff(\mT_{\Sbb^n_+}(g(\overline x))\cap \overline\Gamma^{\perp})&=&\left\{\overline{P}B\overline{P}^T\in\Sbb^n\mid B_{\beta\gamma}=0,\ B_{\gamma\gamma}=0\right\}. \label{eq:aff-critical-cone-SDP}
\end{eqnarray}
For the notational simplicity, we further define the linear operator $\overline{\mB}:\mX\to\Sbb^n$ by 
	\begin{equation}\label{eq:def-B-map}
		\overline\mB(\dxx)=\overline{P}^T(g'(\overline x)\dxx)\overline{P},\quad \dxx\in\mathbb{X}.
\end{equation}
Thus, together with \eqref{eq:critical-2}, we know from \eqref{eq:critical-cone-SDP} that the critical cone $\mC(\overline x)$ of the NLSDP \eqref{prog:NLSDP} at the stationary point $\overline x$ takes the following form:
\begin{equation}\label{eq:C-set-SDP}
	\mC(\overline x)=\left\{\dxx\in\mathbb{X}\mid h'(\overline x)\dxx=0,\ \overline\mB(\dxx)_{\beta\beta}\succeq 0,\ \overline\mB(\dxx)_{\beta\gamma}=0,\ \overline\mB(\dxx)_{\gamma\gamma}=0\right\}.
\end{equation}
It follows from \eqref{eq:aff-critical-cone-SDP} and \eqref{eq:lin-tangent-SDP} that for the stationary point $\overline{x}\in\mathbb{X}$ of \eqref{prog:NLSDP} with a Lagrange multiplier $(\overline\xi,\overline\Gamma)\in M(\overline{x})$, we have
\begin{equation}\label{eq:affC} 
	\app(\overline x,\overline\xi,\overline\Gamma)=\left\{\dxx\in\mathbb{X}\mid h'(\overline x)\dxx=0,\ \overline\mB(\dxx)_{\beta\gamma}= 0,\  \overline\mB(\dxx)_{\gamma\gamma}= 0\right\} 
\end{equation}
and
\begin{equation}\label{eq:linC}
	\appl(\overline x,\overline\xi,\overline\Gamma)=\left\{\dxx\in\mathbb{X}\mid h'(\overline x)\dxx=0,\ \overline\mB(\dxx)_{\beta\beta}= 0,\ \overline\mB(\dxx)_{\beta\gamma}= 0,\ \overline\mB(\dxx)_{\gamma\gamma}= 0\right\},
\end{equation}
 where the sets $\app(\overline x,\overline\xi,\overline\Gamma)$ and $\appl(\overline x,\overline\xi,\overline\Gamma)$ are defined by \eqref{eq:def-app} and \eqref{eq:def-appl}, respectively. Moreover, by employing the explicit formulas \eqref{eq:aff-critical-cone-SDP} and \eqref{eq:lin-tangent-SDP}, we obtain the following useful characterizations of the W-SRCQ (Definition \ref{def:W-SRCQ}) and 
{\tvio the constraint nondegeneracy} (Definition \ref{def:constraint nondegeneracy}) for the NLSDP \eqref{prog:NLSDP}. {\tvio We omit the detailed proof here.} 
\begin{lemma}\label{lemma:W-SRCQ-and-CN}
	Let $\overline{x}\in\mathbb{X}$ be a stationary point of \eqref{prog:NLSDP} with a Lagrange multiplier $(\overline\xi,\overline\Gamma)\in M(\overline{x})$. 
	\begin{itemize}
		\item[(i)] The W-SRCQ holds at $\overline x$ with respect to $(\overline\xi,\overline\Gamma)$ if and only if
		\begin{equation}\label{eq:equal:W-SRCQ}
			\left\{
			\begin{array}{l}
				h'(\overline x)^* \dxi+\overline{\mB}^*\dGG=0,\\ [3pt]
				\dGG_{\alpha\alpha}=0,\ \dGG_{\alpha\beta}=0, \  \dGG_{\alpha\gamma}=0, \  \dGG_{\beta\beta}=0
			\end{array}
			\right.
			\quad\Longrightarrow\quad
			(\dxi,\dGG)=0.
		\end{equation}
		\item[(ii)] The constraint nondegeneracy holds at $\overline{x}$ if and only if
		\begin{equation}\label{eq:equal:CNCQ}
			\left\{
			\begin{array}{l}
				h'(\overline x)^* \dxi+\overline{\mB}^*\dGG=0,\\ [3pt]
				\dGG_{\alpha\alpha}=0,\ \dGG_{\alpha\beta}=0,\ \dGG_{\alpha\gamma}=0 
			\end{array}
			\right.
			\quad\Longrightarrow\quad
			(\dxi,\dGG)=0.
		\end{equation}
	\end{itemize}
\end{lemma}

For the stationary point $\overline{x}\in\mathbb{X}$ of the NLSDP \eqref{prog:NLSDP} with a Lagrange multiplier $(\overline{\xi},\overline{\Gamma})\in M(\overline{x})$, the support function of the second order tangent set in {\tvio the second order conditions} (Definitions \ref{def:SOSC}, \ref{def:S-SOSC} and \ref{def:W-SOC}) has the following explicit characterization \cite[Lemma 3.1]{sun2006strong}:
\begin{equation}\label{eq:sigma}
\sigma(\overline y,\mathcal{T}^2_{K}(G(\overline x),G'(\overline x)\dxx))
=\sigma(\overline\Gamma,\mathcal{T}^2_{\Sbb^n_+}(g(\overline x),g'(\overline x)\dxx))
=2\sum_{i\in\alpha}\sum_{j\in \gamma}\frac{\overline\lambda_j}{\overline\lambda_i}\overline\mB(\dxx)_{ij}^2,\quad \dxx\in \mathbb{X},
\end{equation}
where $K=\{0\}\times\mathbb{S}^n_+$, $G(x)=(h(x),g(x))$ and $\overline{y}=(\overline{\xi},\overline{\Gamma})$.

Define $\mZ:=\mX\times\RR^m\times\Sbb^n$. Let $(\overline{x},\overline{\xi},\overline{\Gamma})$ be a KKT solution of the NLSDP \eqref{prog:NLSDP}. Consider the  Bouligand generalized Jacobian $\partial_BF(\overline x,\overline\xi,\overline\Gamma)$ and Clarke generalized Jacobian $\partial_CF(\overline x,\overline\xi,\overline\Gamma)$ of the KKT nonsmooth mapping $F$ given by \eqref{eq:SDP-natural-mapping}. It follows from \cite[Lemma 2.1]{sun2006strong} (see also \cite[Lemma 1]{chan2008constraint}) that $\overline{\cal U}\in \partial_BF(\overline x,\overline\xi,\overline\Gamma)$ ($\partial_CF(\overline x,\overline\xi,\overline\Gamma)$) if and only if there exists $\overline{\cal V}\in \partial_B\Pi_{\Sbb^m_+}(g(\overline x)+\overline\Gamma)$ ($\partial_C\Pi_{\Sbb^m_+}(g(\overline x)+\overline\Gamma)$) such that for any $(\dxx,\dxi,\dGG)\in \mZ$,
\begin{equation}\label{eq:UandV}
\overline{\cal U}(\dxx,\dxi,\dGG)=
\begin{bmatrix}
\nabla_{xx}^2L(\overline x,\overline \xi,\overline\Gamma)\dxx+h'(\overline x)^*\dxi +g'(\overline x)^*\dGG \\
h'(\overline x)\dxx \\
-g'(\overline x)\dxx+\overline{\cal V}(g'(\overline x)\dxx+\dGG) 
\end{bmatrix}.
\end{equation} 
In particular, let $\overline{\cal V}_{0}$ and $\overline{\cal V}_{\mI}$ be two mappings defined by \eqref{eq:exactV0} and \eqref{eq:exactV1} for $A=g(\overline x)+\overline\Gamma$, respectively. Define the following two elements $\overline{\cal U}_{0}$ and $\overline{\cal U}_{\mI}$ in $\partial_BF(\overline x,\overline\xi,\overline\Gamma)$ with  $\overline{\cal V}_{0}$ and $\overline{\cal V}_{\mI}$ by \eqref{eq:UandV}, respectively, i.e., for any $(\dxx,\dxi,\dGG)\in \mZ$,
\begin{equation}\label{eq:U0}
	\overline{\cal U}_{0}(\dxx,\dxi,\dGG)=
	\begin{bmatrix}
		\nabla_{xx}^2L(\overline x,\overline \xi,\overline\Gamma)\dxx+h'(\overline x)^*\dxi +g'(\overline x)^*\dGG \\
		h'(\overline x)\dxx \\
		-g'(\overline x)\dxx+\overline{\cal V}_0(g'(\overline x)\dxx+\dGG)
	\end{bmatrix}
\end{equation}  
and
\begin{equation}\label{eq:U1}
	\overline{\cal U}_{\cal I}(\dxx,\dxi,\dGG)=
	\begin{bmatrix}
		\nabla_{xx}^2L(\overline x,\overline \xi,\overline\Gamma)\dxx+h'(\overline x)^*\dxi +g'(\overline x)^*\dGG \\
		h'(\overline x)\dxx \\
		-g'(\overline x)\dxx+\overline{\cal V}_{\mI}(g'(\overline x)\dxx+\dGG)
	\end{bmatrix}.
\end{equation}
In the following two subsections, we {\tvio will study the existence} of {\tvio nonsingular elements} in generalized Jacobians at a KKT point $(\overline x,\overline \xi,\overline\Gamma)$.  

\subsection{Sufficient conditions for the existence of {\tvio nonsingular elements} in generalized Jacobians}

In this subsection, we study the sufficient conditions to ensure the existence of {\tvio nonsingular elements} in generalized Jacobians of the KKT nonsmooth mapping $F$ \eqref{eq:SDP-natural-mapping}. The following results illustrate that under the W-SOC and constraint nondegeneracy (or the S-SOSC and W-SRCQ), the mapping $\overline{\cal U}_{0}$ (or $\overline{\cal U}_{\mI}$) defined in \eqref{eq:U0} (or \eqref{eq:U1}) is indeed a nonsingular element of the  Bouligand generalized Jacobian of $F$. 
\begin{theorem}\label{theorem:bar-N-SDP}
 Let $\overline{x}\in\mathbb{X}$ be a stationary point of \eqref{prog:NLSDP} with a Lagrange multiplier $(\overline\xi,\overline\Gamma)\in M(\overline{x})$.
	\begin{itemize}
		\item[(i)] If the W-SOC and constraint nondegeneracy  hold at $(\overline x,\overline \xi,\overline\Gamma)$, then $\overline{\cal U}_{0}\in \partial_B F(\overline x,\overline \xi,\overline\Gamma)$ in \eqref{eq:U0} is nonsingular.
		
		\item[(ii)] If the S-SOSC and W-SRCQ hold at $(\overline x,\overline \xi,\overline\Gamma)$, then $\overline{\cal U}_{\mI}\in \partial_B F(\overline x,\overline \xi,\overline\Gamma)$ in \eqref{eq:U1} is nonsingular.
	\end{itemize}
\end{theorem}
\begin{proof}
	Let $A=g(\overline{x})+\overline{\Gamma}$ satisfy the eigenvalue decomposition \eqref{eq:eig-decomp-SDP} and $\overline{P}\in{\cal O}^n(A)$. We only show (i) holds, as (ii) can be established through a similar argument.

		Suppose that $\overline{\cal U}_{0}\in \partial_B F(\overline x,\overline \xi,\overline\Gamma)$ is singular, i.e., there is $0\neq (\dxx,\dxi,\dGG)\in\mZ$ such that $\overline{\cal U}_{0}(\dxx,\dxi,\dGG)=0$. It then follows from \eqref{eq:U0} and Proposition \ref{prop:dif-B-F} that
	\begin{eqnarray}
		\nabla_{xx}^2L(\overline x,\overline \xi,\overline\Gamma)\dxx+h'(\overline x)^*\dxi+g'(\overline x)^*\dGG=0,&\label{eq:Vz1-Vzero}\\
		h'(\overline x)\dxx =0,&  \label{eq:h-part} \\
		\widetilde{\Delta\Gamma}_{\alpha\alpha}=0,\quad \widetilde{\Delta \Gamma}_{\alpha\beta}=0,&
		\label{theo:fo-LI-1}\\
		\overline\mB(\dxx)_{\beta\beta}=0,\quad \overline\mB(\dxx)_{\beta\gamma}=0,\quad \overline\mB(\dxx)_{\gamma\gamma}=0,& \label{theo:fo-linNLCQ-1}\\
		(E_{\alpha\gamma}-\overline \Xi_{\alpha\gamma} )\circ \overline\mB(\dxx)_{\alpha\gamma}-\overline \Xi_{\alpha\gamma} \circ \widetilde{\Delta\Gamma}_{\alpha\gamma}=0, & \label{theo:fo-U-1}
	\end{eqnarray}
	where $\widetilde{\Delta\Gamma}:= \overline{P}^T\Delta\Gamma\overline{P}$, the linear operator $\overline{\cal B}:\mathbb{X}\to \mathbb{S}^n$ is defined by \eqref{eq:def-B-map}, the matrix $\overline\Xi\in \mathbb{S}^n$ is given by \eqref{eq:Xi} for $A$ and $E\in\Sbb^n$ denotes the matrix whose elements are all ones. Then, by \eqref{eq:Vz1-Vzero}-\eqref{theo:fo-U-1} and \eqref{eq:sigma}, we obtain that
\begin{align}
	0&=\left\la\dxx ,\nabla_{xx}^2L(\overline x,\overline \xi,\overline\Gamma) \dxx +h'(\overline x) ^*\dxi +g'(\overline x)^*\dGG \right\ra \nn\\
	&=\left\la\dxx ,\nabla_{xx}^2L(\overline x,\overline \xi,\overline\Gamma) \dxx \right\ra+\left\la h'(\overline x) \dxx ,\dxi \right\ra+\left\la\overline\mB(\dxx), \widetilde{\Delta \Gamma}\right\ra \nn\\
	&=\la\dxx ,\nabla_{xx}^2L(\overline x,\overline \xi,\overline\Gamma) \dxx \ra+2\sum_{i\in\alpha}\sum_{j\in \gamma}\frac{-\overline\lambda_j}{\overline\lambda_i}(\overline\mB(\dxx)_{ij})^2  \nn\\
	&=\la\dxx ,\nabla_{xx}^2L(\overline x,\overline \xi,\overline\Gamma) \dxx \ra- \sigma(\overline\Gamma,\mathcal{T}^2_{\Sbb^n_+}(g(\overline x),g'(\overline x)\dxx)). \label{eq:cont-SOSC-bar}
\end{align}
Meanwhile, by \eqref{theo:fo-linNLCQ-1} and \eqref{eq:h-part}, we know from \eqref{eq:linC} that $\dxx\in\appl(\overline x,\overline \xi,\overline\Gamma)$. Since the W-SOC holds at $(\overline x,\overline \xi,\overline\Gamma)$, we know from \eqref{eq:cont-SOSC-bar} that $\dxx=0$. It then follows from  \eqref{theo:fo-U-1}  that $\widetilde{\Delta\Gamma}_{\alpha\gamma}=0$. This, together with \eqref{theo:fo-LI-1} and \eqref{eq:Vz1-Vzero}, yields that
\begin{equation*} 
\left\{
\begin{array}{l}
h'(\overline x)^*\dxi+\overline\mB^*\widetilde{\Delta\Gamma}=0,\\
\widetilde{\Delta\Gamma}_{\alpha\alpha}=0,\quad \widetilde{\Delta \Gamma}_{\alpha\beta}=0,\quad \widetilde{\Delta\Gamma}_{\alpha\gamma}=0.
\end{array}
\right.
\end{equation*}
Since the constraint nondegeneracy holds at $(\overline x,\overline \xi,\overline\Gamma)$, we know from \eqref{eq:equal:CNCQ} in Lemma \ref{lemma:W-SRCQ-and-CN} (ii) that  $(\dxi,\Delta\Gamma)=0$, which contradicts with $(\dxx,\dxi,\dGG)\neq 0$. The proof is then completed. $\hfill \Box$
\end{proof}

\begin{remark}
It should be noted that Theorem \ref{theorem:bar-N-SDP} can be easily extended to cases where $K$ in \eqref{prog:P} is a Cartesian product of a finite number of positive semidefinite cones and zero vectors. By regarding a polyhedron as a Cartesian product of a finite number of one-dimensional positive semidefinite cones and zero vectors, the results derived in Theorem \ref{theorem:bar-N-SDP} are consistent with the findings presented in \cite[P644]{izmailov2003karush}.
\end{remark}

It is natural to question whether the W-SRCQ and  W-SOC are sufficient for the existence of nonsingular elements in the Bouligand generalized Jacobian $\partial_B F(\overline x,\overline \xi,\overline\Gamma)$. We address this question with the following proposition, which confirms the sufficiency under the condition that $|\beta|\le 1$, i.e., the number of zero eigenvalues of $g(\overline x)+\overline\Gamma$ is less than or equal to one.

\begin{proposition}\label{prop:beta<=1}
	Let $\overline{x}\in\mathbb{X}$ be a stationary point of \eqref{prog:NLSDP} with a Lagrange multiplier $(\overline\xi,\overline\Gamma)\in M(\overline{x})$. Suppose that $|\beta|\le 1$. If the W-SOC and W-SRCQ hold at $(\overline x,\overline \xi,\overline\Gamma)$, then at least one of $\{\overline{\cal U}_{0},\overline{\cal U}_{\mI}\}$ is nonsingular.
\end{proposition}
\begin{proof} 
	Consider the following two cases.
	
	\noindent{\bf Case 1}: $|\beta|=0$. In this case, from Lemma \ref{lemma:W-SRCQ-and-CN}, we know that the W-SRCQ is equivalent to the constraint nondegeneracy. Moreover, \eqref{eq:affC} and \eqref{eq:linC} imply that $\app(\overline x, \overline \xi, \overline\Gamma) = \appl(\overline x, \overline \xi, \overline\Gamma)$, and thus further imply that the W-SOC is equivalent to the S-SOSC.
	We also have from Proposition \ref{prop:dif-B-F} that $\partial_B F(\overline x,\overline \xi,\overline\Gamma)=\{ F'(\overline x,\overline \xi,\overline\Gamma)\}=\{\overline {\cal U}_{0}\}=\{\overline {\cal U}_{\mI}\}$. The desired result follows directly from Theorem \ref{theorem:bar-N-SDP}.
	
	\noindent{\bf Case 2}: $|\beta|=1$. Since the W-SOC holds at $(\overline x,\overline \xi,\overline\Gamma)$, we know from (i) of Theorem \ref{theorem:bar-N-SDP} that if the {\tvio constraint nondegeneracy holds} at $(\overline x,\overline \xi,\overline\Gamma)$, then $\overline{\cal U}_{0}$ is nonsingular. Suppose that the constraint nondegeneracy does not hold at $(\overline x,\overline \xi,\overline\Gamma)$. Since the W-SRCQ holds at $(\overline x,\overline \xi,\overline\Gamma)$, we know from Lemma \ref{lemma:W-SRCQ-and-CN} that there exists $(\dxi,\dGG) \neq 0$ such that
	\begin{equation*}
		\left\{
		\begin{array}{l}
			h'(\overline x)^* \dxi+\overline{\mB}^*\dGG=0,\\ [3pt]
			\dGG_{\alpha\alpha}=0,\ \dGG_{\alpha\beta}=0,\ \dGG_{\alpha\gamma}=0,  \\
			\dGG_{\beta\beta}\neq 0.
		\end{array} 
		\right.
	\end{equation*} 
	Next, we show that  $\app(\overline x,\overline\xi,\overline\Gamma)=\appl(\overline x,\overline\xi,\overline\Gamma)$ with their explicit expressions given in \eqref{eq:affC} and \eqref{eq:linC}, respectively.  Clearly, $\appl(\overline x,\overline \xi,\overline\Gamma)\subseteq \app(\overline x,\overline \xi,\overline\Gamma)$. For any $\dxx\in\app(\overline x,\overline \xi,\overline\Gamma)$, we have
	\begin{align}
		0=&\la \dxx,h'(\overline x)^* \dxi+\overline{\mB}^*\dGG\ra=\la h'(\overline x)\dxx, \dxi\ra+\la \overline{\mB}(\dxx),\dGG\ra=\la \overline{\mB}(\dxx)_{\beta\beta},\dGG_{\beta\beta}\ra,
	\end{align}
	where the last equation follows from \eqref{eq:affC}. Since $|\beta|=1$, by noting that $\dGG_{\beta\beta}\neq 0$, we know that $\overline{\mB}(\dxx)_{\beta\beta}=0$. Therefore,  we obtain from \eqref{eq:linC} that $\appl(\overline x,\overline \xi,\overline\Gamma)\subseteq \app(\overline x,\overline \xi,\overline\Gamma)$. Then, by Definitions \ref{def:W-SOC} and \ref{def:S-SOSC},   the W-SOC and S-SOSC are equivalent. Therefore, we know from (ii) of Theorem \ref{theorem:bar-N-SDP} that $\overline{\cal U}_{\mI}$ is nonsingular. The proof is then completed.	$\hfill \Box$
\end{proof}

\begin{remark}\label{rema:SC}
	The condition $|\beta|=0$ is satisfied if and only if the well-known strict complementarity condition holds for \eqref{prog:NLSDP}. In this case, it can be verified that the sets in \eqref{eq:critical-cone-SDP}, \eqref{eq:lin-tangent-SDP}, and \eqref{eq:aff-critical-cone-SDP} are identical, and the sets in \eqref{eq:affC} and \eqref{eq:affC} are likewise identical. Therefore, the W-SRCQ is equivalent to the constraint nondegeneracy, and the W-SOC is equivalent to the S-SOSC. Moreover, in this scenario, the mapping $F$ is smooth near $(\overline x,\overline\xi,\overline\Gamma)$, reducing our conditions to those applicable to the classical Newton method. 
\end{remark}

\begin{remark}\label{rema:w-w-NLP-B}
 	It is worth noting that the condition $|\beta|\le 1$ is always satisfied for nonlinear programming (NLP). Actually, as we mentioned before, the positive orthant $\mathbb{R}^n_+$ can be represented as the Cartesian product of $n$ one-dimensional positive semidefinite cones. Consequently, based on Proposition \ref{prop:beta<=1}, the existence of a nonsingular element in the Bouligand generalized Jacobian $\partial_B F(\overline x,\overline \xi,\overline\Gamma)$ is ensured under the W-SOC and the W-SRCQ for the conventional NLP, which recovers the corresponding result obtained by Izmailov and Solodov {\tvio \cite[Proposition 6]{izmailov2003karush}}.
\end{remark}
	
However, the following simple example illustrates that for the NLSDP \eqref{prog:NLSDP}, if the number of zero eigenvalues of $g(\overline x)+\overline\Gamma$ is greater than one, i.e., $|\beta|>1$, the W-SRCQ and W-SOC are not sufficient to guarantee the existence of nonsingular elements in $\partial_B F(\overline x,\overline \xi,\overline\Gamma)$. This demonstrates the fundamental differences between the NLSDP and the NLP.
\begin{example}\label{example:non-SRCQ}
Consider the following SDP problem:
\begin{align}\nn
\begin{aligned}
\min\quad     &   0 \\
\mbox{s.t.}\quad & x_{12}=0, \\
& X\in \Sbb^2_+ , 
\end{aligned}
\end{align}
where $X=\begin{bmatrix} x_{11} &x_{12}\\x_{21}&x_{22}\end{bmatrix}\in\Sbb^2$. It is clear that $(\overline X,\overline \xi, \overline\Gamma)=(0,0,0)$ is a KKT point. We have $\beta=\{1,2\}$ and $\alpha=\gamma=\emptyset$. It is easy to verify that the W-SRCQ and W-SOC hold at $(\overline X,\overline \xi,\overline\Gamma)$. Moreover, it follows from \cite[Lemma 11]{pang2003semismooth} that ${\cal U}\in\partial_B F(\overline X,\overline \xi,\overline\Gamma)$ if and only if there exists $\Omega\in R$ such that for any $(\dXX,\dxi,\dGG)\in \Sbb^2\times\RR\times\Sbb^2$,
\begin{equation}\label{eq:def-U-exmaple}
	{\cal U}(\dXX,\dxi,\dGG)=
	\begin{bmatrix}
		\begin{bmatrix}
			0 & 0 \\
			0  &0 
		\end{bmatrix}+
		\begin{bmatrix}
			0 & \dxi \\
			\dxi  &0
		\end{bmatrix}+\dGG\\
		\dXX_{12}\\
		- \dXX+\Omega\circ(\dXX+\dGG)
	\end{bmatrix},
\end{equation}
where  $R\in \mathbb{S}^n$ is the set of matrices defined by
\[
R:=\left\{\begin{bmatrix} 0 &0\\0&0\end{bmatrix},\begin{bmatrix} 1 &1\\1&1\end{bmatrix} \right\}\bigcup\left\{\begin{bmatrix} 0 &t\\t&1\end{bmatrix}\mid t\in[0,1]\right\}\bigcup\left\{\begin{bmatrix} 1 &t\\t&0\end{bmatrix}\mid t\in[0,1]\right\}.
\]
Then, one can check easily that all elements in $\partial_B F(\overline X,\overline \xi,\overline\Gamma)$ are singular. 
\end{example}

Next, let us consider the Clarke generalized Jacobian defined by \eqref{eq:partialC} of the KKT nonsmooth mapping \eqref{eq:SDP-natural-mapping}. Interestingly, one may observe that for Example \ref{example:non-SRCQ}, the convex combination $t {\cal U}_0+(1-t){\cal U}_{\cal I}\in \partial_C F(\overline X,\overline \xi,\overline\Gamma)$ for any $t\in(0,1)$ is actually nonsingular, where ${\cal U}_0$ and  ${\cal U}_{\cal I}$ are elements in $\partial_B F(\overline X,\overline \xi,\overline\Gamma)$ defined by \eqref{eq:def-U-exmaple} with $\Omega=\begin{bmatrix}
		0 & 0 \\
		0  &0 
	\end{bmatrix}$ and $\begin{bmatrix}
	1 & 1 \\
	1 & 1 
	\end{bmatrix}$, respectively. In fact, we have the following general results, which demonstrate that the W-SOC and W-SRCQ guarantee the existence of a nonsingular element in the Clarke generalized Jacobian.

\begin{proposition}\label{prop:pC}
{Let $\overline{x}\in\mathbb{X}$ be a stationary point of \eqref{prog:NLSDP} with a Lagrange multiplier $(\overline\xi,\overline\Gamma)\in M(\overline{x})$. Suppose that the W-SOC and W-SRCQ hold at $(\overline x,\overline \xi,\overline\Gamma)$. Let $\overline{\cal U}_0$ and $\overline{\cal U}_{\cal I}$ be two elements defined in \eqref{eq:U0} and \eqref{eq:U1}, respectively.}
\begin{itemize}
\item[(i)] There exists $t\in(0,1)$ such that
\begin{equation*}
t\overline{\cal U}_{0}+(1-t)\overline{\cal U}_{\mI}\in\partial_C F(\overline x,\overline \xi,\overline\Gamma) \text{ is nonsingular.} 
\end{equation*}

\item[(ii)] If, additionally, the NLSDP \eqref{prog:NLSDP} is convex, that is, it becomes a CSDP, then for any \(t \in (0,1)\),  
\begin{equation*}
	t\overline{\cal U}_{0}+(1-t)\overline{\cal U}_{\mI}\in\partial_C F(\overline x,\overline \xi,\overline\Gamma) \text{ is nonsingular.}
\end{equation*}
\end{itemize}
\end{proposition}
\begin{proof} Let $A=g(\overline{x})+\overline{\Gamma}$ satisfy the eigenvalue decomposition \eqref{eq:eig-decomp-SDP} and $\overline{P}\in{\cal O}^n(A)$ with the index sets $\alpha$, $\beta$ and $\gamma$ defined by \eqref{eq:index}. Recall the linear operator $\overline{\cal B}:\mathbb{X}\to \mathbb{S}^n$ defined in \eqref{eq:def-B-map}.
	
\noindent {\bf (i)} Assume that for any $t\in (0,1)$, $t\overline{\cal U}_{0}+(1-t)\overline{\cal U}_{\mI}\in\partial_C F(\overline x,\overline \xi,\overline\Gamma)$ is singular. Define a sequence of  elements $\{\overline{\cal U}^k\}$ as $\overline{\cal U}^k=\frac{k}{k+1}\overline{\cal U}_{0}+\frac{1}{k+1}\overline{\cal U}_{\mI}\in\partial_C F(\overline x,\overline \xi,\overline\Gamma)$ for each $k$. Since $\overline{\cal U}^k$ is singular, we know that for each $k$, there are $0\neq(\dxx^k,\dxi^k,\dGG^k)\in\mZ$ such that $\overline{\cal U}^k(\dxx^k,\dxi^k,\dGG^k)=0$. Thus, for each $k$, we know from \eqref{eq:UandV} and Proposition \ref{prop:dif-B-F} that  
\begin{align}
\nabla_{xx}^2L(\overline x,\overline \xi,\overline\Gamma)\dxx^k+h'(\overline x)^*\dxi^k+g'(\overline x)^*\dGG^k=0,&\label{eq:Uk=0-1}\\
h'(\overline x)\dxx^k =0,& \label{eq:Uk=0-2}\\
\widetilde{\Delta\Gamma}^k_{\alpha\alpha}=0,\quad \widetilde{\Delta \Gamma}^k_{\alpha\beta}=0,&
\label{eq:Uk=0-3}\\
\overline{\mB}(\dxx^k)_{\beta\gamma}=0,\quad \overline{\mB}(\dxx^k)_{\gamma\gamma}=0,& \label{eq:Uk=0-4}\\
(E_{\alpha\gamma}-\overline{\Xi}_{\alpha\gamma} )\circ \overline{\mB}(\dxx^k)_{\alpha\gamma}-\overline{\Xi}_{\alpha\gamma} \circ \widetilde{\Delta\Gamma}^k_{\alpha\gamma}=0, & \label{eq:Uk=0-5}\\
k\overline{\mB}(\dxx^k)_{\beta\beta}-\widetilde{\Delta\Gamma}^k_{\beta\beta}=0, & \label{eq:Uk=0-6}
\end{align}
where  $\widetilde{\Delta\Gamma}^k =\overline{P}^T\Delta\Gamma^k\overline{P}$ with $\overline{P}\in\mathcal{O}^n(A)$ and $A=g(\overline{x})+\overline{\Gamma}$, the matrix $\overline\Xi\in \mathbb{S}^n$ is given by \eqref{eq:Xi} and $E\in\Sbb^n$ denotes the matrix whose elements are all ones. We then conclude that for each $k$, $\dxx^k\neq 0$. Indeed, if there exists some $\overline{k}$ such that $\dxx^{\overline{k}}=0$, then by \eqref{eq:Uk=0-3}, \eqref{eq:Uk=0-5} and \eqref{eq:Uk=0-6}, we obtain  that 
\begin{equation}\label{eq:pC-dx0-1}
\widetilde{\Delta\Gamma}^{\overline{k}}_{\alpha\alpha}=0,\quad \widetilde{\Delta\Gamma}^{\overline{k}}_{\alpha\beta}=0, \quad \widetilde{\Delta\Gamma}^{\overline{k}}_{\beta\beta}=0  \quad {\rm and} \quad\widetilde{\Delta\Gamma}^{\overline{k}}_{\alpha\gamma}=0.
\end{equation} 
Thus, combining with \eqref{eq:Uk=0-1}, we know that 
\begin{equation}\label{eq:pC-dx0-2}
\left\{
\begin{array}{l}
h'(\overline x)^*\dxi^{\overline{k}}+\overline\mB^*\dtG^{\overline{k}}=0, \\ [3pt]
\widetilde{\Delta\Gamma}^{\overline{k}}_{\alpha\alpha}=0,\quad  \widetilde{\Delta\Gamma}^{\overline{k}}_{\alpha\beta}=0,\quad \widetilde{\Delta\Gamma}^{\overline{k}}_{\beta\beta}=0,\quad \widetilde{\Delta\Gamma}^{\overline{k}}_{\alpha\gamma}=0.
\end{array}
\right.
\end{equation} 
It then follows from \eqref{eq:equal:W-SRCQ} in Lemma \ref{lemma:W-SRCQ-and-CN} (i) that $(\dxi^{\overline{k}},{\Delta\Gamma}^{\overline{k}})=0$, 
which, together with $\Delta x^{\overline k} = 0$, contradicts with the assumption. Now, since $\dxx^k\neq 0$ for all $k$, without loss of generality, we may assume that $\|\dxx^k\|=1$ and $\dxx^k\to\dxx^{\infty}\neq 0$ as $k\to\infty$. By taking the inner product with $\dxx^k$ in \eqref{eq:Uk=0-1}, we obtain from \eqref{eq:Uk=0-2}-\eqref{eq:Uk=0-6} that for each $k$,
\begin{align}
\la\dxx^k ,\nabla_{xx}^2L(\overline x,\overline \xi,\overline\Gamma) \dxx^k \ra+2\sum_{i\in\alpha}\sum_{j\in \gamma}\frac{-\overline\lambda_j }{\overline\lambda_i }\overline\mB(\dxx^k)_{ij}^2+\sum_{i\in\beta}\sum_{j\in\beta}k\overline\mB(\dxx^k)_{ij}^2 =0. \label{prop-eq:Vz-1-dx}
\end{align}
Let $\mQ: \mX\to\mX$ be the linear map such that
\begin{align}\nn
\la d,\mQ(d)\ra=\la d ,\nabla_{xx}^2L(\overline x,\overline \xi,\overline\Gamma) d \ra+2\sum_{i\in\alpha}\sum_{j\in \gamma}\frac{-\overline\lambda_j }{\overline\lambda_i }\overline\mB(d)_{ij}^2, \quad \forall \, d\in \mX.
\end{align} 
Then there exists $q>0$ that $|\mQ(d)|\le q\| d\| ^2$ for any $d\in\mX$. Therefore, we have from \eqref{prop-eq:Vz-1-dx} that for each $k$,
\begin{align}\nn
|\overline\mB(\dxx^k)_{ij}|^2\le \frac{1}{k} q,\quad (i,j)\in \beta\times\beta.
\end{align}
As a result, $\overline\mB(\dxx^{\infty})_{\beta\beta}=0$, which, together with \eqref{eq:Uk=0-4}, implies $\dxx^{\infty}\in\appl(\overline x,\overline \xi,\overline\Gamma)$. 
 Further note that \eqref{prop-eq:Vz-1-dx} implies that
	\begin{align}\nonumber
		\langle \dxx^k, \mQ \dxx^k\rangle \le 0.
	\end{align}
	Taking the limit as \(k \to +\infty\), we obtain \(\langle \dxx^\infty, \mQ \dxx^\infty\rangle \le 0\), which contradicts the W-SOC given that \(0 \neq \dxx^{\infty} \in \appl(\overline x, \overline \xi, \overline\Gamma)\).
Therefore, there exists a nonsingular $\overline{\cal U}^k$.

\noindent {\bf (ii)} For the CSDP, if there is $t\in(0,1)$ that $\overline{\cal U}_t=t\overline {\cal U}_{0}+(1-t)\overline {\cal U}_{\mI}$ is singular. Then, there is $(\dxx,\dxi,\dGG)\neq0$ such that $\overline{\cal U}_t(\dxx,\dxi,\dGG)=0$. Similarly to the proof of (i), we obtain from \eqref{eq:pC-dx0-1}-\eqref{prop-eq:Vz-1-dx} that $\dxx\neq0$ and 
\begin{align}\label{prop-eq:Vz-1-dx:22}
\la\dxx ,\nabla_{xx}^2L(\overline x,\overline \xi,\overline\Gamma) \dxx \ra+2\sum_{i\in\alpha}\sum_{j\in \gamma}\frac{-\overline\lambda_j }{\overline\lambda_i }\overline\mB(\dxx)_{ij}^2+\sum_{i\in\beta}\sum_{j\in\beta}\frac{t}{1-t}\overline\mB(\dxx)_{ij}^2 =0. 
\end{align}
By convexity, we know that $\nabla_{xx}^2L(\overline x,\overline \xi,\overline\Gamma)$ is positive semidefinite. Thus, $\overline\mB(\dxx)_{\beta\beta}=0$. This, together with \eqref{eq:Uk=0-2} and \eqref{eq:Uk=0-4}, yields $0\neq \dxx\in\appl(\overline x,\overline \xi,\overline\Gamma)$, and {\tvio further implies} the contradiction between \eqref{prop-eq:Vz-1-dx:22} the W-SOC. Therefore, we conclude that $\overline{\cal U}_t=t\overline {\cal U}_{0}+(1-t)\overline {\cal U}_{\mI}$ is nonsingular for any $t\in(0,1)$. $\hfill \Box$
\end{proof}

\begin{remark}\label{rema:w-w-SDP-C}
	According to the proof of part (i) of Proposition \ref{prop:pC}, there exists only a finite number of singular elements on the line segment connecting $\overline{\cal U}_0$ and $\overline{\cal U}_{\mI}$. In fact, by considering the matrix representations of $\overline{\cal U}_0$ and $\overline{\cal U}_{\mI}$ and invoking the fundamental theorem of algebra, we are able to deduce that the non-Zero polynomial ${\rm det}(t\overline{\cal U}_{0}+(1-t)\overline{\cal U}_{\mI})$ has a number of roots that does not exceed its degree.
\end{remark}

\subsection{Necessary conditions for the existence of nonsingular elements in generalized Jacobians}
 
	We first show that when the NLSDP \eqref{prog:NLSDP} is convex, i.e., for the CSDP, the conditions of Theorem \ref{theorem:bar-N-SDP} are actually necessary. {\vio The proof relies on the following lemma regarding the characterization of W-SOC violations.

\begin{lemma}\label{lemma:W-SOC-vio}
	Let $\overline{x}\in\mathbb{X}$ be a stationary point of the NLSDP with a Lagrange multiplier $(\overline\xi,\overline\Gamma)\in M(\overline{x})$. Suppose that the W-SOC does not hold at $(\overline x,\overline \xi,\overline\Gamma)$ and $\nabla_{xx}^2L(\overline x,\overline \xi,\overline\Gamma)$ is positive semidefinite, there exists $\dxx\neq 0$ that  
\begin{align}\label{eq:WSOC_fail}
\left\{
\begin{aligned}
&\nabla_{xx}^2L(\overline x,\overline \xi,\overline\Gamma)\dxx=0,\\
&h'(\overline x) \dxx=0,\\
&\overline\mB(\dxx)_{\alpha\gamma}=0, \quad \overline\mB(\dxx)_{\beta\beta}= 0,\quad \overline\mB(\dxx)_{\beta\gamma}= 0,\quad \overline\mB(\dxx)_{\gamma\gamma}= 0.
\end{aligned}
\right.
\end{align}
\end{lemma}

	The proof is omitted here for the sake of simplicity. Based on this lemma,} we have the following characterizations of the nonsingularity of $\overline{\cal U}_0$ and $\overline{\cal U}_{\cal I}$ defined by \eqref{eq:U0} and \eqref{eq:U1}, respectively, for the CSDP. 
\begin{proposition}\label{prop:pd:LSDP-V01}
Let $\overline{x}\in\mathbb{X}$ be a stationary point of the CSDP with a Lagrange multiplier $(\overline\xi,\overline\Gamma)\in M(\overline{x})$.
\begin{itemize}
\item[(i)] The W-SOC and constraint nondegeneracy hold  at $(\overline x,\overline \xi,\overline\Gamma)$ if and only if $\overline{\cal U}_0$ defined in \eqref{eq:U0} is nonsingular;

\item[(ii)] the S-SOSC and W-SRCQ hold  at $(\overline x,\overline \xi,\overline\Gamma)$ if and only if  $\overline{\cal U}_{\cal I}$ defined in \eqref{eq:U1} is nonsingular.
\end{itemize}
\end{proposition}
\begin{proof}
Let $A=g(\overline{x})+\overline{\Gamma}$ satisfy the eigenvalue decomposition \eqref{eq:eig-decomp-SDP} and $\overline{P}\in{\cal O}^n(A)$ with the index sets $\alpha$, $\beta$ and $\gamma$ defined by \eqref{eq:index}. Recall the linear operator $\overline{\cal B}:\mathbb{X}\to \mathbb{S}^n$ defined by \eqref{eq:def-B-map}. We only show (i), as (ii) can be established through a similar argument.

By Theorem \ref{theorem:bar-N-SDP}, we only need to show that the nonsingularity of $\overline{\cal U}_0$ implies the W-SOC and constraint nondegeneracy. Suppose $\overline{\cal U}_0(\dxx,\dxi,\dGG)=0$. We know from \eqref{eq:U0} and Proposition \ref{prop:dif-B-F} that
\begin{align}
\nabla_{xx}^2L(\overline x,\overline \xi,\overline\Gamma)\dxx+h'(\overline x)^*\dxi+g'(\overline x)^*\dGG=0,&\label{eq:U0=1-1}\\
h'(\overline x)\dxx =0,\label{eq:U0=1-2}& \\
\widetilde{\Delta\Gamma}_{\alpha\alpha}=0,\quad \widetilde{\Delta\Gamma}_{\alpha\beta}=0,&
\label{eq:U0=1-3}\\
\overline\mB(\dxx)_{\beta\beta}=0,\quad \overline\mB(\dxx)_{\beta\gamma}=0,\quad \overline\mB(\dxx)_{\gamma\gamma}=0,& \label{eq:U0=1-4}\\
(E_{\alpha\gamma}-\overline\Xi_{\alpha\gamma} )\circ \overline\mB(\dxx)_{\alpha\gamma}-\overline \Xi_{\alpha\gamma} \circ \widetilde{\Delta\Gamma}_{\alpha\gamma}=0, & \label{eq:U0=1-5}
\end{align}
where $\widetilde{\Delta\Gamma}=\overline{P}^T\dGG\overline{P} $, $\overline\Xi$ is defined in \eqref{eq:Xi} for $A$ and $E$ represents the matrix consisting of all ones in $\Sbb^n$. Suppose that the constraint nondegeneracy does not hold. Then, by (ii) of Lemma \ref{lemma:W-SRCQ-and-CN}, we know that there exist $0\neq(\dxi,\dGG)$ such that  
\begin{align}\nn
	\left\{
	\begin{array}{l}
		h'(\overline x)^*\dxi+\overline\mB^*\widetilde{\Delta\Gamma}=0,\\
		\widetilde{\Delta\Gamma}_{\alpha\alpha}=0,\quad  \widetilde{\Delta\Gamma}_{\alpha\beta}=0, \quad  \widetilde{\Delta\Gamma}_{\alpha\gamma}=0.
	\end{array}
	\right.
\end{align}
Thus, it follows from \eqref{eq:U0=1-1}-\eqref{eq:U0=1-5} that $(\dxx,\dxi,\dGG)=(0,\dxi,\dGG)\neq 0$ satisfy $\overline{\cal U}_0(\dxx,\dxi,\dGG)=0$, which contradicts with the nonsingularity of $\overline{\cal U}_0$. Thus, the constraint nondegeneracy holds at $(\overline x,\overline \xi,\overline\Gamma)$.

On the other hand, suppose that the W-SOC does not hold at $(\overline x,\overline \xi,\overline\Gamma)$. Since $\nabla_{xx}^2L(\overline x,\overline \xi,\overline\Gamma)$ is positive semidefinite, by Lemma \ref{lemma:W-SOC-vio}, there exists $\dxx\neq 0$ such that \eqref{eq:WSOC_fail} holds.	
It then yields that $\overline{\cal U}_0(\dxx,0,0)=0$,  which again contradicts with the nonsingularity of $\overline{\cal U}_0$. The proof of part (i) has been established. $\hfill \Box$
\end{proof}

\begin{remark}
	In the context of linear semidefinite programming, a particular instance of the CSDP, Chan and Sun \cite[Proposition 17]{chan2008constraint} demonstrated that the nonsingularity of $\overline{\cal U}_{0}$ ($\overline{\cal U}_{\cal I}$) implies primal (dual) constraint nondegeneracy. In contrast, Proposition \ref{prop:pd:LSDP-V01} in the current work provides a complete characterization of the nonsingularity of $\overline{\cal U}_{0}$ and $\overline{\cal U}_{\cal I}$ for more general CSDP.
\end{remark}

\begin{remark} 
	For the general NLSDP, one can still use the first part of the proof to show that the nonsingularity of $\overline{\cal U}_0$ (or $\overline{\cal U}_{\mI}$) implies the constraint nondegeneracy (or the W-SRCQ).
\end{remark}

Finally, we can establish in the forthcoming proposition that, for the general NLSDP, the W-SRCQ is necessary to ensure the existence of nonsingular elements in the Clarke generalized Jacobian. Additionally, for the CSDP, both the W-SRCQ and the W-SOC {\tvio are requisites}.
 
\begin{proposition}\label{prop:non-to-CQ}
	Let $\overline{x}\in\mathbb{X}$ be a stationary point of the NLSDP \eqref{prog:NLSDP} with a Lagrange multiplier $(\overline\xi,\overline\Gamma)\in M(\overline{x})$.
\begin{itemize}
\item[(i)] If there is a nonsingular element $\overline{\cal U}\in \partial_C F(\overline x,\overline \xi,\overline \Gamma)$, then the W-SRCQ holds at $(\overline x,\overline \xi,\overline \Gamma)$.

\item[(ii)] For the CSDP, if there is a nonsingular element $\overline{\cal U}\in \partial_C F(\overline x,\overline \xi,\overline \Gamma)$, then the W-SRCQ and W-SOC hold at $(\overline x,\overline \xi,\overline \Gamma)$.
\end{itemize}
\end{proposition}
\begin{proof} Let $A=g(\overline{x})+\overline{\Gamma}$ satisfy the eigenvalue decomposition \eqref{eq:eig-decomp-SDP} and $\overline{P}\in{\cal O}^n(A)$ with the index sets $\alpha$, $\beta$ and $\gamma$ defined by \eqref{eq:index}. Recall the linear operator $\overline{\cal B}:\mathbb{X}\to \mathbb{S}^n$ defined by \eqref{eq:def-B-map}.
	
\noindent {\bf (i)} Let $(\dxi,\Delta\Gamma)\in\mathbb{R}^m\times\mathbb{S}^{n}$ be a solution to the following system: 
\begin{equation*} 
	\left\{
	\begin{aligned}
		&h'(\overline x)^*\dxi+\overline\mB^*\widetilde{\Delta\Gamma}=0,\\
		&\widetilde{\Delta\Gamma}_{\beta\beta}=0,\quad \widetilde{\Delta\Gamma}_{\alpha\alpha}=0, \quad \widetilde{\Delta\Gamma}_{\alpha\beta}=0,\quad \widetilde{\Delta\Gamma}_{\alpha\gamma}=0,
	\end{aligned}
	\right.
\end{equation*}
where $\widetilde{\Delta\Gamma}=\overline{P}^T\Delta\Gamma\overline{P}$. Then, we know from 
\eqref{eq:UandV} and Proposition \ref{prop:dif-B-F} that  
\begin{align} \label{eq:any-U}
	\overline{\cal U}(0,\dxi,\dGG)=
	\begin{bmatrix}
		h'(\overline x)^*\dxi+\overline\mB^*\widetilde{\Delta\Gamma} \\
		0\\
		\overline{P} 
		\begin{bmatrix}
			\widetilde{\Delta\Gamma}_{\alpha\alpha}&\widetilde{\Delta\Gamma}_{\alpha\beta}&\overline\Xi_{\alpha\gamma}\circ\widetilde{\Delta\Gamma}_{\alpha\gamma} \\
			\widetilde{\Delta\Gamma}_{\beta\alpha}&{\cal W}(\widetilde{\Delta\Gamma}_{\beta\beta})&0 \\
			\overline\Xi_{\gamma\alpha}\circ\widetilde{\Delta\Gamma}_{\gamma\alpha}&0&0 
		\end{bmatrix}
		\overline{P}^T
	\end{bmatrix}=0,
\end{align}
where ${\cal W}\in\partial_C\Pi_{\Sbb^{|\beta|}_+}(0)$. Since  $\overline{\cal U}$ is nonsingular, we obtain that $(\dxi,\Delta\Gamma)=0$. Therefore, we know from \eqref{eq:equal:W-SRCQ} in Lemma \ref{lemma:W-SRCQ-and-CN} (i) that the W-SRCQ holds at $(\overline x,\overline \xi,\overline \Gamma)$.

\noindent {\bf (ii)} We only need to show that the W-SOC (Definition \eqref{def:W-SOC}) holds at $(\overline x,\overline \xi,\overline\Gamma)$. Suppose on the contrary that the W-SOC does not hold. Since $\nabla_{xx}^2L(\overline x,\overline \xi,\overline\Gamma)$ is positive semidefinite, by Lemma \ref{lemma:W-SOC-vio}, there exists $\dxx\neq 0$ such that \eqref{eq:WSOC_fail} holds. 
It, together with \eqref{eq:any-U}, implies that
\begin{align}\nn
	\overline{\cal U}(\dxx,0,0)=
	\begin{bmatrix}
		\nabla_{xx}^2L(\overline x,\overline \xi,\overline\Gamma)\dxx \\
		h'(\overline x)\dxx \\
		-\overline{P}(\overline\mB\dxx)\overline{P}^T+\overline{P}
		\begin{bmatrix}
			(\overline\mB\dxx)_{\alpha\alpha}&(\overline\mB\dxx)_{\alpha\beta}&\overline\Xi_{\alpha\gamma}\circ(\overline\mB\dxx)_{\alpha\gamma} \\
			(\overline\mB\dxx)_{\beta\alpha}&W((\overline\mB\dxx)_{\beta\beta})&0 \\
			\overline\Xi_{\gamma\alpha}\circ(\overline\mB\dxx)_{\gamma\alpha}&0&0 
		\end{bmatrix}\overline{P}^T
	\end{bmatrix}=0,
\end{align}
which contradicts with the nonsingularity of $\overline{\cal U}$. The proof is then completed. $\hfil \Box$
\end{proof}

\section{The primal-dual characterizations for convex quadratic semidefinite programming}\label{section:cha}
	In this section, our objective is to investigate the primal-dual characterizations of the W-SRCQ (Definition \ref{def:W-SRCQ}) and W-SOC  (Definition \ref{def:W-SOC}) for the following convex quadratic semidefinite programming (QSDP):
\begin{equation}\label{prog:P-CP}
\begin{array}{cl}
\min\quad        & \displaystyle\frac{1}{2}\la x,\mQ x\ra +\la c,x\ra	 \\
\mbox{s.t.}\quad & \mH x-p=0, \\
& \mG x-q\in\Sbb^n_+ , 
\end{array}
\end{equation}
where $c\in \mX$, $p\in\RR^m$, $q\in\Sbb^n$, the self-adjoint linear operator  $\mQ:\mX\to\mX$ is  positive semi-definite, and $\mH:\mX\to\RR^m$ and $\mG:\mX\to\Sbb^n$ are two given linear operators. The dual problem (in the restricted Wolfe sense \cite{wolfe1961duality,li2018qsdpnal}) of the QSDP \eqref{prog:P-CP} takes the form:
\begin{equation}\label{prog:D-CP}
	\begin{array}{cl}
		\displaystyle\min_{\omega\in \mW,\xi\in\RR^m,\Gamma\in\Sbb^n}\quad& \displaystyle\frac{1}{2}\la \omega,\mQ \omega\ra +\la p,\xi\ra+\la q,\Gamma\ra \\
		\mbox{s.t.}\quad & -\mQ \omega-c-\mH^*\xi-\mG^*\Gamma=0, \\
		& \Gamma\in \Sbb^n_-, 
	\end{array}
\end{equation}
where $\mW\subseteq \mX$ is a linear subspace containing $\tIm(\mQ)$. Thus, the KKT system of \eqref{prog:D-CP} is given by
\begin{align}
\left\{ 
\begin{aligned}
&\begin{bmatrix}
\mQ\omega \\
p \\
q
\end{bmatrix}+
\begin{bmatrix}
-\mQ \\
-\mH \\
-\mG
\end{bmatrix}\tau+
\begin{bmatrix}
0 \\
0\\
\mI
\end{bmatrix}\Upsilon=0, \\
&\mQ \omega+c+\mH^*\xi+\mG^*\Gamma=0, \\
&\Upsilon\in \mathcal{N}_{\Sbb^n_-}(\Gamma), \\
&\omega\in \mW,\quad  \xi\in\RR^m,\quad  \Gamma\in\Sbb^n,\quad  \tau\in\mX,\quad  \Upsilon\in\Sbb^n.\\
\end{aligned}
\right. \label{KKT:prog:D-CP}
\end{align}

\begin{proposition}\label{prop:pd-opc}
	Consider the convex quadratic semidefinite programming \eqref{prog:P-CP} and its dual problem \eqref{prog:D-CP}.
\begin{itemize}
\item[(i)] Suppose that $\mW=\mX$. Let $\overline{x}\in\mathbb{X}$ be a stationary point of the convex quadratic semidefinite programming \eqref{prog:P-CP} with a Lagrange multiplier $(\overline\xi,\overline\Gamma)\in M(\overline{x})$. The primal W-SOC holds at $(\overline x,\overline \xi,\overline\Gamma)$ if and only if the dual W-SRCQ holds at $(\overline x,\overline \xi,\overline\Gamma,\overline x,\mG \overline x-q)$.

\item[(ii)] Suppose that $\mW=\tIm(\mQ)$. Let $(\overline \omega,\overline \xi,\overline \Gamma)\in\mW\times\RR^m\times\Sbb^n$ be a stationary point of the dual problem  \eqref{prog:D-CP} with a Lagrange multiplier $(\overline x,\overline\Upsilon)\in M(\overline \omega,\overline \xi,\overline \Gamma)$. The dual W-SOC holds at $(\overline \omega,\overline \xi,\overline \Gamma,\overline x,\overline\Upsilon)$ if and only if the primal W-SRCQ holds at $(\overline x,\overline \xi,\overline \Gamma)$.
\end{itemize}
\end{proposition}
\begin{proof} We demonstrate (i) here, as (ii) follows from it with only minor modifications.
	
Let $A=g(\overline{x})+\overline{\Gamma} = {\cal G}\overline x - q + \overline\Gamma$ satisfy the eigenvalue decomposition \eqref{eq:eig-decomp-SDP} with $\overline{P}\in{\cal O}^n(A)$, the index sets $\alpha$, $\beta$ and $\gamma$ defined by \eqref{eq:index}, and $\overline{\cal B}:\mathbb{X}\to \mathbb{S}^n$ defined by \eqref{eq:def-B-map}. By the convexity and Lemma \ref{lemma:W-SOC-vio}, we know that the primal W-SOC holds at $(\overline x,\overline \xi,\overline{\Gamma})$ if and only if
\begin{align}
\left\{
\begin{aligned}
&\mQ \dxx=0, \quad \mH \dxx=0,\\
&\overline\mB(\dxx)_{\alpha\gamma}=0, \quad \overline\mB(\dxx)_{\beta\beta}= 0,\quad \overline\mB(\dxx)_{\beta\gamma}= 0, \quad \overline\mB(\dxx)_{\gamma\gamma}= 0  
\end{aligned}
\right.
\quad \Longrightarrow \quad 
\dxx=0. \label{eq:pd-prithe W-SOC}
\end{align}

	{\tvio On the other hand}, it is obvious that $(\overline x,\overline \xi,\overline\Gamma,\overline x,\mG \overline x-q)$ is a KKT point of \eqref{prog:D-CP} by \eqref{KKT:prog:D-CP}. By Definition  \ref{def:W-SRCQ}, we know that the W-SRCQ of the dual problem \eqref{prog:D-CP} at $(\overline x,\overline \xi,\overline\Gamma,\overline x,\mG \overline x-q)$ takes the form: 
\begin{align}\label{eq:d-w-srcq}
\begin{bmatrix}
 \mQ &\mH^* & \mG^* \\
0 & 0 & \mI
\end{bmatrix}
\begin{bmatrix}
\mX\\ \RR^m \\ \Sbb^n
\end{bmatrix}+
\begin{bmatrix}
\{0\} \\ 
\aff(\mT_{\Sbb^n_-}(\overline\Gamma)\cap [\mG \overline x-q]^{\perp})
\end{bmatrix}
=
\begin{bmatrix}
\mX \\
\Sbb^n
\end{bmatrix}.
\end{align}
Since $\aff(\mT_{\Sbb^n_-}(\overline\Gamma)\cap [\mG \overline x-q]^{\perp})=\{\overline{P}B\overline{P}^T\in\Sbb^n\ |\ B_{\alpha\beta}=0,\ B_{\alpha\alpha}= 0\}$, we know that \eqref{eq:d-w-srcq} is equivalent to 
\begin{align}\label{eq:pd-dualWSRCQ}
\Null \Big(\begin{bmatrix}
\mQ & 0\\ \mH& 0\\ \overline\mB&\mI
\end{bmatrix} \Big)
\bigcap
\begin{bmatrix}
\mX \\
\{B\in\Sbb^n\ |\ B_{\alpha\beta}=0,\ B_{\alpha\alpha}= 0\}^{\perp}
\end{bmatrix}
=\{0\}.
\end{align}
Meanwhile, it is not difficult to observe that  \eqref{eq:pd-dualWSRCQ} and \eqref{eq:pd-prithe W-SOC} are equivalent. Consequently, the primal W-SOC is satisfied at $(\overline x, \overline \xi,\overline\Gamma)$ if and only if the dual W-SRCQ is satisfied at $(\overline x, \overline \xi,\overline\Gamma,\overline x, \mG \overline x - q)$. This completes the proof. $\hfill \Box$
\end{proof}

\begin{remark}
	By employing a similar argument, we are able to obtain the primal-dual characterizations of the S-SOSC (Definition \ref{def:S-SOSC}) and the constraint nondegeneracy (Definition \ref{def:constraint nondegeneracy}) for the QSDP \eqref{prog:P-CP}. This extends the corresponding results obtained in Chan and Sun \cite{chan2008constraint} for the linear SDP. Also,  for the QSDP \eqref{prog:P-CP}, a primal-dual characterizations of the SRCQ (Definition \ref{def:SRCQ}) and SOSC (Definition \ref{def:SOSC}) is established by \cite{han2018linear}.
\end{remark}

\section{A semismooth Newton method with correction}\label{section:5}

In this section, based on the theoretical results obtained in previous sections, we introduce a new semismooth Newton method with a correction step, which enjoys the locally quadratically convergence rate even without the generalized Jacobian regularity.
   
Denote $\mathbb{Z}:=\mathbb{X}\times\mathbb{R}^m\times\mathbb{S}^n$. Let $\delta>0$ be some given constant. Define the correction mapping ${\cal P}_{\delta}: \mathbb{Z}\to \mathbb{Z}$ by 
\begin{equation}\label{eq:def-pi-mapping}
{\cal P}_{\delta}(Z):=\big({x},{\xi},{\Gamma}-\sum_{i\in \theta} \lambda_i({A}){P}_{i}{P}_i^T\big), \quad Z=({x},{\xi},{\Gamma})\in\mathbb{Z},
\end{equation}
where $ {A}:=g( x)+ \Gamma$  has the eigenvalue decomposition \eqref{eq:eig-decomp-SDP} with ${P}\in{\cal O}^n({A})$ and $\theta:=\{i\mid |\lambda_i({A})|\le \delta\}.$

Given $Z=(x,\xi,\Gamma)\in \mathbb{Z}$, define two elements ${\cal U}_0$ and ${\cal U}_{\cal I}\in \partial_BF(Z)$ at $Z$ similarly as those in \eqref{eq:U0} and \eqref{eq:U1}, i.e., for any $(\dxx,\dxi,\dGG)\in\mathbb{Z}$,
{\tvio\begin{align}\label{eq:Uk}
{\cal U}_{l}(\dxx,\dxi,\dGG) =
\begin{bmatrix}
\nabla^2_{xx}L(x,\xi,\Gamma)\dxx+ h'(x)^*\dxi+ g'(x)^*\dGG\\
 h'(x)\dxx \\
- g'(x)\dxx+{\cal V}_{l}(g'(x)\dxx+\dGG)
\end{bmatrix}  ,\quad l\in\{0,\mI\}
\end{align}}
where ${\cal V}_0$ and ${\cal V}_{\cal I}\in \partial_B\Pi_{\Sbb^m_+}(A)$ is given by \eqref{eq:exactV0} and \eqref{eq:exactV1} in Proposition \ref{prop:dif-B-F}, respectively. The following proposition addresses the legitimacy of the correction mapping. It essentially shows that under weaker regularity conditions, when $\widehat{Z}$ is sufficiently close to $\overline{Z}$, the corresponding ${\cal U}_{0}$ (or ${\cal U}_{\cal I}$) in the Bouligand generalized Jacobian of $F$ at the corrected point $Z = {\cal P}_{\delta}(\widehat Z)$ with properly chosen $\delta$ is nonsingular and its inverse is uniformly bounded. As one can see later, this will be crucial for the local convergence analysis of the proposed semismooth Newton algorithm.

\begin{proposition}\label{prop:nons} 
Let $\overline{x}\in\mathbb{X}$ be a stationary point of the NLSDP \eqref{prog:NLSDP} with a Lagrange multiplier $(\overline\xi,\overline\Gamma)\in M(\overline{x})$. {Let $\delta$ be any fixed constant in $(0, \min(\overline\lambda_\alpha, |\overline\lambda_\gamma|))$.}
\begin{itemize}
\item[(i)] If the W-SOC (Definition \ref{def:W-SOC}) and constraint nondegeneracy (Definition \ref{def:constraint nondegeneracy}) hold at $\overline{Z}=(\overline{x},\overline\xi,\overline\Gamma)\in\mathbb{Z}$, then there exists $\kappa_{0}>0$ such that for any $\widehat{Z}=(\widehat{x},\widehat{\xi},\widehat\Gamma)$ sufficient close to $\overline{Z}$ and $Z={\cal P}_{\delta}(\widehat{Z})$, the corresponding ${\cal U}_{0}\in\partial_B F(Z)$ defined by \eqref{eq:Uk} with respect to $Z$ is nonsingular and
\begin{equation}\label{eq:Uni-bound-U0}
\kappa_{0}^{-1}\|\Delta Z\| \le \|{\cal U}_{0}(\Delta Z)\|\quad \forall\,\Delta Z\in\mathbb{Z}.
\end{equation}

\item[(ii)] If the S-SOSC (Definition \ref{def:S-SOSC}) and W-SRCQ  (Definition \ref{def:W-SRCQ}) hold at $\overline{Z}=(\overline{x},\overline\xi,\overline\Gamma)\in\mathbb{Z}$, then there exists $\kappa_{1}>0$ such that for any $\widehat{Z}=(\widehat{x},\widehat{\xi},\widehat\Gamma)$ sufficient close to $\overline{Z}$ and $Z={\cal P}_{\delta}(\widehat{Z})$, the corresponding ${\cal U}_{\cal I}\in\partial_B F(Z)$ defined by \eqref{eq:Uk} with respect to $Z$ is nonsingular and
\begin{equation}\label{eq:Uni-bound-U1}
\kappa_{1}^{-1}\|\Delta Z\| \le \|{\cal U}_{\cal I}(\Delta Z)\|\quad \forall\,\Delta Z\in\mathbb{Z}.
\end{equation}
\end{itemize} 
\end{proposition}
\begin{proof} We focus on the result (i), since (ii) can be obtained similarly. 
	
It suffices to show that the inequality \eqref{eq:Uni-bound-U0} holds, as the nonsingularity of ${\cal U}_{0}$ is a direct implication of \eqref{eq:Uni-bound-U0}. Assume, on the contrary,  there exist a sequence $\{\widehat{Z}^{\nu} \}$ converging to $ \overline{Z}$ and a nonzero sequence $\{\Delta Z^{\nu}=(\dxx^{\nu},\dxi^{\nu},\dGG^{\nu})\}\subseteq \mathbb{Z}$ such that ${\cal U}^{\nu}_{0}\in\partial_B F(Z^{\nu})$ defined by \eqref{eq:Uk} 
satisfy 
\[
\|{\cal U}^{\nu}_{0}(\Delta Z^{\nu})\|< 1/\nu  \|\Delta Z^{\nu}\|. 
\] 
Without loss of generality, we may assume that $\|\Delta Z^{\nu}\| =1$ for each $\nu$ and $\Delta Z^{\nu}\to\Delta Z^{\infty}\neq 0$ as $\nu\to \infty$. For each $\nu$, let  $\widehat{A}^{\nu}=g(\widehat{x}^{\nu})+\widehat\Gamma^{\nu}$ enjoy the eigenvalue decomposition \eqref{eq:eig-decomp-SDP} with $\widehat{P}^{\nu}\in {\cal O}^n(\widehat{A}^{\nu})$. For each $\nu$, denote
\[
\zeta^{\nu}:=\{i\mid \widehat\lambda_i^{\nu}\ge \delta\}, \quad \theta^{\nu}:=\{i\mid |\widehat\lambda_i^{\nu}|\le \delta\}\quad \mbox{and}\quad \kappa^{\nu}:=\{i\mid \widehat\lambda_i^{\nu}\le -\delta\},
\]
where $\widehat{\lambda}^{\nu}:=\lambda(\widehat{A}^{\nu})$ {\tvio are the eigenvalues of} $\widehat{A}^{\nu}$. 
Since $\widehat{A}^\nu\to \overline A = g(\overline{x})+\overline{\Gamma}$ as $\nu \to \infty$, we know from the continuity of eigenvalues \cite{lancaster1964eigenvalues} that $\widehat\lambda^\nu \to \lambda(A)$ as $\nu\to \infty$. Notice that $\delta\in(0, \min(\overline\lambda_\alpha, |\overline\lambda_\gamma|))$, for $\nu$ large enough,
\begin{align}\label{eq:zeta}
\zeta^{\nu}\equiv {\alpha}, \quad  \theta^{\nu}\equiv {\beta}  \quad \mbox{and} \quad   \kappa^{\nu}\equiv {\gamma},
\end{align}
where $ {\alpha}$, $ {\beta}$ and $ {\gamma}$ are the index sets defined by \eqref{eq:eig-de} for $\overline{A}$. Since $Z^{\nu}={\cal P}_{\delta}(\widehat{Z}^{\nu})$, it holds that
\begin{equation*}
	\|Z^{\nu}-\widehat{Z}^{\nu}\|=\|{\cal P}_{\delta}(\widehat{Z}^{\nu})-\widehat{Z}^{\nu}\| =\|\sum_{i\in \beta} \lambda_i(\widehat{A}^{\nu})\widehat{P}_{i}^{\nu}(\widehat{P}_i^{\nu})^T\| 
	\le \|  
	\widehat \lambda_\beta^\nu \| \to 0, \mbox{ as } \nu \to \infty.
\end{equation*}
Thus, we know that $Z^{\nu}\to \overline{Z}$ as $\nu \to \infty$. It also holds that for each $\nu$,
\begin{align}\label{eq:Anu}
A^{\nu}=g(x^{\nu})+\Gamma^{\nu} = \widehat{P}^{\nu} 
\begin{bmatrix}
	\widehat\Lambda_{ {\alpha} {\alpha}}^{\nu} & 0 & 0 \\
	0 & 0 & 0 \\
	0 & 0 & \widehat\Lambda_{ {\gamma} {\gamma}}^{\nu}
\end{bmatrix}
(\widehat{P}^{\nu})^T 
\end{align}
with $\widehat{\Lambda}^{\nu}={\rm Diag}(\widehat\lambda^{\nu})$, and $A^\nu \to \overline A$ as $\nu \to \infty$.

Let $\overline{\cal U}_0$ be given by \eqref{eq:U0} with $\overline{\cal V}_0$ defined by \eqref{eq:exactV0}. By noting that $\|{\cal U}_{0}^{\nu}(\Delta Z^{\nu})\|<1/\nu$, $\|\Delta Z^{\nu}\| =1$, $f$, $h$ and $g$ are {\tvio twice continuously differentiable} and $Z^{\nu}\to\overline{Z}$, we have 
\begin{eqnarray}
	&& \| \overline {\cal U}_{0}(\Delta Z^{\infty})\|\le \| \overline {\cal U}_{0}(\Delta Z^{\infty})-{\cal U}_{0}^{\nu}(\Delta Z^{\nu})\|+\|{\cal U}_{0}^{\nu}(\Delta Z^{\nu})\| \nn \\ [3pt]
	&\le& \| \overline {\cal U}_{0}(\Delta Z^{\infty})-\overline{\cal U}_{0}(\Delta Z^{\nu})\| +\| \overline {\cal U}_{0}(\Delta Z^{\nu})-{\cal U}_{0}^{\nu}(\Delta Z^{\nu})\|+1/\nu \nn \\ [3pt]
	&\le&  O(\| \Delta Z^{\infty}-\Delta Z^{\nu}\| )+\| \overline{\cal V}_{0}(g'(\overline{x})\dxx^{\nu}+\dGG^{\nu}) - {\cal V}_{0}^{\nu}(g'(x^{\nu})\dxx^{\nu}+\dGG^{\nu})\|+1/\nu,\label{eq:U0Uk-diff1}
\end{eqnarray}
where for each $\nu$, ${\cal V}_{0}^{\nu}$ is defined by \eqref{eq:exactV0} with respect to $A^{\nu}$. By Proposition \ref{prop:dis-P} and taking subsequence if necessary, we may further assume that  $\widehat{P}^{\nu}\to P^{\infty}\in \mathcal{O}^n(\overline{A})$. Then, \eqref{eq:Anu} and Proposition \ref{prop:dif-B-F} imply that $\lim_{\nu\to\infty}\mathcal{V}^\nu_0 = \overline{\cal V}_0$.
Since ${\cal V}_{0}^{\nu}$ is uniformly bounded \cite[Proposition 1]{meng_semismoothness_2005}, we thus have as $\nu\to \infty$,
\begin{eqnarray}
&&\| {\cal V}_{0}^{\nu}(g'(x^{\nu})\dxx^{\nu}+\dGG^{\nu})-\overline{\cal V}_{0}(g'(\overline x)\dxx^{\nu}+\dGG^{\nu})\| \nn \\ [3pt]
&\le& \| {\cal V}_{0}^{\nu}(g'(x^{\nu})\dxx^{\nu}+\dGG^{\nu})-{\cal V}_{0}^{\nu}(g'(\overline x)\dxx^{\nu}+\dGG^{\nu})\|  \nn \\ [3pt]
&&+ \| {\cal V}_{0}^{\nu}(g'(\overline x)\dxx^{\nu}+\dGG^{\nu})-\overline{\cal V}_{0}(g'(\overline x)\dxx^{\nu}+\dGG^{\nu})\|  \nn\\ [3pt]
&\le&O(\| (g'(x^{\nu})-g'(\overline x))\dxx^{\nu}\| )+o(\| g'(\overline x)\dxx^{\nu}+\dGG^{\nu}\| ) =o(1).\label{eq:U0Uk-diff2}
\end{eqnarray}
Combining \eqref{eq:U0Uk-diff1} and \eqref{eq:U0Uk-diff2}, we obtain $\| \overline {\cal U}_{0}(\Delta Z^{\infty})\| =0$, which implies that $\overline{\cal U}_{0}$ is singular. {\tvio This contradicts Theorem \ref{theorem:bar-N-SDP} (i)}, since the W-SOC and constraint nondegeneracy are assumed. The proof is then completed.
$\hfill \Box$ 
\end{proof}

Using Proposition \ref{prop:nons}, we can further establish a local error bound for the nonsmooth map $F$ at the corrected points around $\overline{Z}$ in the following proposition.

\begin{proposition}\label{prop:error-bound}
Let $\overline{x}\in\mathbb{X}$ be a stationary point of the NLSDP \eqref{prog:NLSDP} with a Lagrange multiplier $(\overline\xi,\overline\Gamma)\in M(\overline{x})$ and the constant $\delta\in(0, \min(\overline\lambda_\alpha, |\overline\lambda_\gamma|)).$ Then, for any given $\varrho >0$, there exists
a neighborhood $\cal N$ of $\overline Z$ such that
for any $\widehat Z\in {\cal N}$, $Z = {\cal P}_{\delta}(\widehat Z)$, and $U \in \partial_C F(Z)$
\begin{equation}\label{eq:semi-neighborhood0}
	\| F(Z) - F(\overline{Z}) - {\cal U}(Z - \overline{Z}) \| \le \varrho \| Z - \overline{Z} \|.
\end{equation}
{Recall the definitions of constants $\kappa_0$ and $\kappa_1$ in Proposition \ref{prop:nons}.}  It further holds that
\begin{itemize}
\item[(i)] if the W-SOC and the constraint nondegeneracy hold at $\overline Z=(\overline x,\overline\xi,\overline\Gamma)\in\mZ$, then there exists a neighborhood ${\cal N}$ of $\overline{Z}$ such that for any $\widehat Z\in {\cal N}$ and $Z = {\cal P}_\delta(\widehat Z)$, 
\[\| F(Z)\|\ge \frac{1}{2\kappa_{0}}\| Z-\overline Z\|;\]
\item[(ii)] if the S-SOSC and the W-SRCQ hold at $\overline Z=(\overline x,\overline\xi,\overline\Gamma)\in\mZ$, then there exists a neighborhood ${\cal N}$ of $\overline{Z}$ such that for any $\widehat Z\in {\cal N}$ and $Z = {\cal P}_\delta(\widehat Z)$, 
\[\| F(Z)\|\ge \frac{1}{2\kappa_{1}}\| Z-\overline Z\|.\]
\end{itemize}
\end{proposition}
\begin{proof}
	
By the semismoothness of $\Pi_{\Sbb^n_+}(\cdot)$ on $\Sbb^n$ and the twice continuous differentiability of $f$, $g$ and $h$, we know $F$ is semismooth at $\overline Z$. This implies that for any given $\varrho >0$, we can find a neighborhood ${\cal N}_1$ of $\overline Z$  such that for any $Z\in {\cal N}_1$ and $U\in \partial_C F(Z)$
\begin{equation}\label{eq:semi-neighborhood-1}
	\| F(Z) - F(\overline{Z}) - {\cal U}(Z - \overline{Z}) \|	\le \varrho \| Z - \overline{Z} \|.
\end{equation}	
Recall the eigenvalue decomposition of $\overline A = g(\overline x) + \overline{\Gamma}$ in \eqref{eq:eig-decomp-SDP} with index sets $\alpha, \beta$ and $\gamma$ defined in \eqref{eq:index}. For any $\widehat Z$ with $Z = {\cal P}_{\delta}(\widehat Z)$, let $\widehat A :=  g(\widehat{x}) + \widehat{\Gamma}$ satisfy the {\tvio eigenvalue decomposition} \eqref{eq:eig-de} with $\widehat{P}\in {\cal O}^n(\widehat{A})$ and $\widehat{\Lambda} = \Lambda(\widehat{A})$. Then, similar to \eqref{eq:zeta} and \eqref{eq:Anu}, there exists a neighborhood ${\cal N}$ of $\overline Z$ such that for any $\widehat Z\in \cal N$
\begin{align}\nn
	A := g(x)+\Gamma = \widehat{P} 
	\begin{bmatrix}
		\widehat\Lambda_{ {\alpha} {\alpha}} & 0 & 0 \\
		0 & 0 & 0 \\
		0 & 0 & \widehat\Lambda_{ {\gamma} {\gamma}}
	\end{bmatrix}
	(\widehat{P})^T.
\end{align}
Thus,
\begin{align}
	\| Z -\overline Z\| \le & \| Z -\widehat Z\|+\| \widehat Z - \overline Z\|  \le\| \widehat\Lambda_{\beta\beta}\|+\| \widehat Z-\overline Z\| 
	\le  \|\lambda(\widehat A) - \lambda(\overline A) \| + \| \widehat Z - \overline Z\|  \nn\\
	\le& \| g(\widehat{x}) + \widehat \Gamma - g(\overline x) - \overline \Gamma \| + \| \widehat Z - \overline Z\|   
	\le  { (L_g+2) \| \widehat Z - \overline Z\|},  \label{eq:ZZhatbound}
\end{align}
where the last inequality is due to the Lipschitz continuity of $g$ over $\cal N$. From \eqref{eq:ZZhatbound}, by shrinking ${\cal N}$ if necessary, we can assert that $Z = P_\delta(\widehat Z) \in {\cal N}_1$ for any $\widehat Z\in {\cal N}$. This, together with \eqref{eq:semi-neighborhood-1}, implies the desired inequality \eqref{eq:semi-neighborhood0}.

Next, we focus on the rest results (i) and (ii). Here, only the proof of (i) is provided, as (ii) can be proved via similar arguments. 

From \eqref{eq:semi-neighborhood0} and Proposition \ref{prop:nons} (i), we know that there exists a neighborhood $\cal N$ of $\overline{Z}$ such that for any $\widehat Z \in {\cal N}$ and $Z = {\cal P}_\delta(\widehat Z)$, 
\begin{align} \nn
\| F(Z)-{\cal U}_{0}(Z-\overline Z)\|\le \frac{1}{2\kappa_{0}}\| Z-\overline Z\| \quad \mbox{ and } \quad  \| {\cal U}_{0}\|\ge \frac{1}{\kappa_0} ,
\end{align}
where $\mathcal{U}_{0}\in \partial_BF(Z)$ is given in \eqref{eq:Uk}. Thus,
\begin{align}\nn
\|F(Z)\|\ge \|{\cal U}_{0}(Z-\overline Z)\|-\| F(Z)-{\cal U}_{0}(Z-\overline Z)\|\ge  \frac{1}{\kappa_{0}}\|Z-\overline Z\|-\frac{1}{2\kappa_{0}}\| Z-\overline Z\|=\frac{1}{2\kappa_{0}}\| Z-\overline Z\|.
\end{align}
The proof is completed.
$\hfill \Box$ 
\end{proof}

Now we are ready to present our inexact semismooth Newton method with correction in Algorithm \ref{algorithm:ASN} for solving the nonsmooth system $F(Z) = 0$ with $F$ given in \eqref{eq:SDP-natural-mapping}. 
As one can observe, the algorithm is quite similar to classical inexact semismooth Newton method \cite{facchinei2003finite}. The main difference is that in each iteration, an extra correction step, i.e., step 6 in Algorithm \ref{algorithm:ASN}, involving the correction mapping ${\cal P}_{\delta}$ is taken. This correction step, based on Proposition \ref{prop:nons}, ensures locally the nonsingularity of ${\cal U}_0^k$ (or ${\cal U}_{\cal I}^k$), and hence the well-definedness of the proposed algorithm. Meanwhile, it is worth noting that the additional computational costs brought by the correction step are negligible, as the eigendecompostion of $\widehat{A}^{k+1} = g(\widehat x^{k+1}) + \widehat{\Gamma}^{k+1}$ in step 6 can be reused to obtain ${\cal U}^{k+1} \in \partial_B F(Z^{k+1})$ in step 3. Moreover, the inexact computation in step 4 indicates that our proposed algorithm has has the potential for solving large-scale problems.

\begin{algorithm}
\caption{The semismooth Newton method with correction using ${\cal U}_0$ (or ${\cal U}_{\mI}$)}
\label{algorithm:ASN}
\begin{algorithmic}[1]
\State{Initialize $\widehat{Z}^{0}=(\widehat{x}^0,\widehat \xi^0,\widehat\Gamma^0)$, $\delta>0$, $\eta \ge 0$, $\tau\in(0,1]$ and $Z^0={\cal P}_{\delta}(\widehat{Z}^{0})$.}
\For{$k=0,1,\ldots$}
\State{Get ${\cal U}^k={\cal U}^k_0 \in \partial_B F(Z^k)$ by \eqref{eq:Uk}  (${\cal U}^k={\cal U}^k_{\mI} \in \partial_B F(Z^k)$ by \eqref{eq:Uk}).}
\State{Find an approximate solution $d^k$ to the linear system 
$
{\cal U}^k d + F(Z^k) = 0
$
such that \[\| {\cal U}^k d^k + F(Z^k) \| \le  {\min(\eta, \|F(Z^k)\|^{\tau})}  \|F(Z^k)\|.\]
}
\State{$\widehat Z^{k+1}=Z^k+d^k$.}
\State{$Z^{k+1}={\cal P}_{\delta}(\widehat Z^{k+1})$.}
\EndFor
\end{algorithmic}
\end{algorithm}

As promised, in the following theorem, we establish the well-definedness and local convergence results of Algorithm \ref{algorithm:ASN} without assuming the generalized Jacobian regularity.

\begin{theorem}\label{theo:alg-convergence}
Let $\overline{x}\in\mathbb{X}$ be a stationary point of the NLSDP \eqref{prog:NLSDP} with a Lagrange multiplier $(\overline\xi,\overline\Gamma)\in M(\overline{x})$. Let $\delta\in(0, \min(\overline\lambda_\alpha, |\overline\lambda_\gamma|))$. Then, 
\begin{itemize}
\item[(i)] if the W-SOC and the constraint nondegeneracy hold at $\overline Z=(\overline x,\overline\xi,\overline\Gamma)\in\mZ$, then  {there exists $\overline\eta >0$ such that for any $\widehat Z^0$ sufficiently close to $\overline Z$ and $\eta \le \overline \eta$}, Algorithm \ref{algorithm:ASN} using ${\cal U}_{0}$ is well defined and the sequence $\{Z^k\}$ so generated converges to $\overline Z$ superlinearly;

\item[(ii)] if the S-SOSC and W-SRCQ hold at $\overline Z=(\overline x,\overline\xi,\overline\Gamma)\in\mZ$, then there exists $\overline\eta >0$ such that for any $\widehat Z^0$ sufficiently close to $\overline Z$ and $\eta \le \overline \eta$, Algorithm \ref{algorithm:ASN} using ${\cal U}_{\mI}$ is well-defined and the sequence $\{Z^k\}$ so generated converges to $\overline Z$ superlinearly.
\end{itemize}
Moreover, if $f$, $g$ and $h$ are twice Lipschitz continuously differentiable, i.e., they have Lipschitz continuous second order differentials, and $\tau = 1$, then the corresponding convergence rates obtained in (i) and (ii) are quadratic.
\end{theorem}
\begin{proof}
Here, we focus on the case (i) where the W-SOC and the constraint nondegeneracy hold at $\overline Z$ and ${\cal U}_0$ is employed. The proof for case (ii) where the S-SOSC and W-SRCQ hold at $\overline Z$ can be obtained similarly.

 By Proposition \ref{prop:nons}, Proposition \ref{prop:error-bound} and \eqref{eq:ZZhatbound}, we can find a positive constant $\varsigma$, a neighborhood $\mathbb{B}({\overline{Z}},\varsigma)$ of $\overline{Z}$ and {\tvio a positive constant $\kappa_0$} such that for any  $ \widehat Z\in \mathbb{B}({\overline{Z}},\varsigma)$ and $Z = {\cal P}_{\delta}(\widehat Z)$, 
\begin{equation}\label{eq:semi-neighborhood}
\| F(Z) - F(\overline{Z}) - {\cal U}_0(Z - \overline{Z}) \|	\le \frac{1}{  4(L_g+2)\kappa_0} \| Z - \overline{Z} \| \quad  \mbox{ and } \quad \|({\cal U}_0)^{-1} \| \le \kappa_0,
\end{equation}
and
\begin{align}
	\| Z -\overline Z\| \le  
	 { (L_g+2) \| \widehat Z - \overline Z\|},  \label{eq:con-b1}
\end{align}
where ${\cal U}_0 \in \partial_B F(Z)$ is given by \eqref{eq:Uk} and $L_g > 0$ is the Lipschitz constant of $g$ over $\mathbb{B}({\overline{Z}},\varsigma)$. 

    Now, let $\widehat Z^0$ be any point in ${\mathbb{B}}({\overline{Z}}, \varsigma)$. Recall that $Z^0 = {\cal P}_{\delta}(\widehat Z^0)$. 
Then, \eqref{eq:semi-neighborhood} implies that ${\cal U}_0^0\in \partial_B F(Z^0)$ is nonsingular, and therefore $d^0$ and $\widehat Z^1$ are well-defined. Since $F(\overline{Z}) = 0$, we further have 
\begin{align}
	\| \widehat Z^{1}-\overline Z\|  &=\| Z^0-\overline Z-({\cal U}_{0}^0)^{-1} F(Z^0) + d^0 + ({\cal U}_{0}^0)^{-1} F(Z^0) \| \nn\\
	&\le \| -({\cal U}_{0}^0)^{-1}\big[F(Z^0)-F(\overline Z)-{\cal U}_{0}^0(Z^0-\overline Z) - {\cal U}_{0}^0 d^0 - F(Z^0)\big] \|  \nn\\
	& \le \kappa_0|| F(Z^0)-F(\overline Z)-{\cal U}_{0}^0(Z^0-\overline Z) \| + { \eta\kappa_0 \| F(Z^0) - F(\overline Z)\|} \nn\\
	&\le \big(\frac{1}{ 4(L_g + 2)} + \eta\kappa_0 L_F\big) \| Z^0-\overline Z\|, \label{eq:con-b0}
\end{align}
where the last inequality follows from \eqref{eq:semi-neighborhood} and the Lipschitz continuity of $F$ over $\mathbb{B}(\overline{Z},\varsigma)$ with $L_F$ being the corresponding Lipschitz constant. Meanwhile, we know from \eqref{eq:con-b1} that $ \|
Z^0 - \overline{Z} \| \le (L_g+2) \| \widehat Z^0 - \overline{Z} \|\le (L_g + 2)\varsigma$.
Then, for any $0 \le \eta \le \overline\eta = 1/(4\kappa_0L_F(L_g + 2))$, we see from  \eqref{eq:con-b0} that  $ \| \widehat{Z}^1 - \overline{Z}\| \le  \| Z^0-\overline Z\|/(2(L_g + 2)) \le \varsigma $ and 
\begin{equation}\label{eq:contraction_Z} 
\| Z^1 - \overline{Z}\| \le (L_g + 2)\|\widehat{Z}^1 - \overline{Z}\| \le \frac{1}{2}\|Z^0 - \overline{Z}\|.
\end{equation}
In summary, the derivation above shows that if $\widehat Z^0$ belongs to ${\mathbb{B}}({\overline{Z}}, \varsigma)$, then every $\widehat{Z}^k$ and $Z^k$ produced by Algorithm \ref{algorithm:ASN} satisfy
\[ \widehat{Z}^k \in {\mathbb{B}}({\overline{Z}}, \varsigma) \quad \mbox{ and } \quad Z^k \in \mathbb{B}({\overline Z}, (L_g+2)\varsigma), \]
$\{Z^k\}$ converges to $\overline{Z}$ due to the contraction inequality $\eqref{eq:contraction_Z}$. Moreover, the semismoothness of $F$,  \eqref{eq:semi-neighborhood} and \eqref{eq:con-b0} further imply that
\begin{equation}\label{eq:semismooth-hatZkp1}
	\begin{aligned}
\|\widehat Z^{k+1} - \overline{Z} \| = {}&	\| -({\cal U}_{0}^k)^{-1}[F(Z^k)-F(\overline Z)-{\cal U}_{0}^k(Z^k-\overline Z) - {\cal U}_{0}^k d^k - F(Z^k)]\| \\
={}& o(\| Z^k - \overline Z\|) + O(\|Z^k - \overline{Z}\|^{1 + \tau}).
	\end{aligned}
\end{equation}
This, together with the fact that $ \|Z^{k+1} - \overline{Z}\| \le (L_g + 2) \|\widehat{Z}^{k+1} - \overline{Z}\|$, implies the superlinear convergence of $\{Z^k\}$.

If $f,g,h$ are twice Lipschitz continuously differentiable, by the strong semismoothness of $\Pi_{\Sbb^n_+}(\cdot)$ on $\Sbb^n$ and {\tvio the Lipschitz continuous differentiability} of $\nabla_x L$, $g$ and $h$, we know $F$ is strongly semismooth at $\overline Z$. 
The strong semismoothness of $F$ and $\tau = 1$ allow us to replace \eqref{eq:semismooth-hatZkp1} by
the following:
\begin{align}
	\|\widehat Z^{k+1} - \overline{Z} \| = {}&	\| -({\cal U}_{0}^k)^{-1}[F(Z^k)-F(\overline Z)-{\cal U}_{0}^k(Z^k-\overline Z) - {\cal U}_{0}^k d^k - F(Z^k)]\| \nn  \\ 
	={}& O(\|Z^k - \overline{Z}\|^2). \nn
\end{align}
Then, the quadratic convergence of $\{Z^k\}$ follows directly. $\hfill \Box$ 
\end{proof}

\begin{remark}\label{rmk:termination}
Let $\{Z^k\}$ be the sequence generated by Algorithm \ref{algorithm:ASN}. For a given tolerance $\varepsilon >0$, Proposition \ref{prop:error-bound} implies that 
$F(Z^k) < \varepsilon$ is a reasonable stopping criterion for Algorithm \ref{algorithm:ASN}.
In fact, the local Lipschitz continuity of $F$ around $\overline{Z}$ and Proposition \ref{prop:error-bound} imply that $F(Z^k) = \Theta(\|Z^k - \overline{Z}\|)$. 
\end{remark}

\begin{remark}[Choice of $\widehat Z^0$ and $\delta$]
	It is not difficult to observe from Theorem \ref{theo:alg-convergence} and its proof that the choice of $\widehat Z^0$ and $\delta$ is crucial to the performance of Algorithm \ref{algorithm:ASN}. A larger $\delta$ may lead to a more restrictive choice of $\widehat Z^0$ in order to ensure the nonsingularity of ${\cal U}^0$. Meanwhile, accurately determining the value of $\min(\overline \lambda_\alpha, |\lambda_\gamma|)$ for the general NLSDP is often challenging in practical settings, though obtaining a rough estimate may be feasible. In this study, our primary emphasis lies in the local convergence analysis of Algorithm \ref{algorithm:ASN}. We defer the considerations regarding the efficient and robust implementation of the algorithm, including the selection of $\widehat Z^0$ and $\delta$, to be explored in future research.
\end{remark} 

{\vio
To conclude this section, we will illustrate with a simple example why the correction step is essential for semismooth Newton methods. Specifically, without the correction step, the classical semismooth Newton method may fail to converge or even fail to execute due to the singularity of the Jacobian. Moreover, the example illustrates that our algorithm can handle some differentiable but degenerate points, which is beyond the reach of the classical semismooth Newton method.

\begin{example}\label{example:cor-important}
Consider the following polyhedral problem, which can be viewed as a special instance of semidefinite programming:
\begin{equation*}\label{prob:K-lin-proj}
    \begin{array}{cl}
        \min & \displaystyle \frac{1}{2} \| x - (0,1,0)^T \|^2 \\ [5pt]
        \text{s.t.} & \begin{bmatrix} 1 & 1 & 0 \\ 1 & 0 & 1 \end{bmatrix} x = \begin{bmatrix} 1 \\ 0 \end{bmatrix}, \\ [2pt]
        & x \in \mathbb{R}^3_+,
    \end{array}
\end{equation*}
where it is evident that $(0,1,0)^T$ is the unique solution with the multiplier set $M = \{ \left((0,t)^T, (t,0,t)^T\right) \mid t \leq 0 \}$.

Verification reveals that the W-SRCQ and S-SOSC hold at the point 
\begin{align}\nonumber
(\overline{x}, \overline{\xi}, \overline{\Gamma})= \left((0,1,0)^T, (0,0)^T, (0,0,0)^T\right),
\end{align}
and the associated $\overline{\mU}_{\mI}$ is nonsingular. Therefore, one can employ Algorithm \ref{algorithm:ASN} using $\mathcal{U}_{\mI}$ to solve this problem. However,  for any   point $(x^0, \xi^0, \Gamma^0)$ such that
\begin{equation}\nonumber
      0< \|(x^0, \xi^0, \Gamma^0)-(\overline{x}, \overline{\xi}, \overline{\Gamma})\|\le 1,\ x^0_1+\Gamma^0_1<0,\ x^0_2+\Gamma^0_2>0 \mbox{ and } x^0_3+\Gamma^0_3<0,
\end{equation} 
which is in proximity to $(\overline{x}, \overline{\xi}, \overline{\Gamma})$, the corresponding $F$ is differentiable at $(x^0, \xi^0, \Gamma^0)$ with a singular Jacobian $\nabla F(x^0, \xi^0, \Gamma^0)$. This suggests that the classical semismooth Newton method will encounter failure at such $(x^0, \xi^0, \Gamma^0)$. However, our proposed Algorithm \ref{algorithm:ASN} is capable of handling such scenarios through the correction. Taking 
\begin{align*}
Z^0=(x^0, \xi^0, \Gamma^0)=\left((-\varepsilon,1+\varepsilon,0)^T, (0,\varepsilon)^T, (0,0,-\varepsilon)^T\right),\quad \varepsilon\in(0,1/10),
\end{align*}
as an example, Algorithm \ref{algorithm:ASN} (with $\delta=1/5$) will first correct it to 
\begin{align*}
\widehat Z^0=(\widehat x^0,\widehat  \xi^0,\widehat  \Gamma^0)=\left((-\varepsilon,1+\varepsilon,0)^T, (0,\varepsilon)^T, (\varepsilon,0,0)^T\right).
\end{align*}
Then the Newton step based on $\mathcal{U}_{\mI}$ reads
\begin{align*}
( x^1,  \xi^1,  \Gamma^1)=\widehat Z^0-(\mathcal{U}^0_{\mI})^{-1}F(\widehat Z^0)=\left((0,1,0)^T, (0,0)^T, (0,0,0)^T\right)=(\overline{x}, \overline{\xi}, \overline{\Gamma}),
\end{align*}
which indicates that our algorithm converges to a KKT solution in only one iteration.
\end{example}}

\begin{remark}\vio
	In the recent study \cite{duy2023generalized}, the authors introduced a Newton-type algorithm to solve the subgradient inclusion problem, achieving a superlinear convergence rate assuming metric regularity, but without requiring strong metric regularity. Even for tilt-stable local minimizers, their results improve on those of Mordukhovich and Sarabi \cite{mordukhovich2021generalized}. 

However, in the context of optimization problems, the equivalence between the metric regularity and the strong metric regularity at local optimal solutions has been established for the nonlinear programming in \cite{dontchev1996characterizations},  and for the nonlinear second-order cone programming and NLSDP in the recent preprints \cite{chen2024aubin} and \cite{chen2024equivalent}, respectively. 
Meanwhile, the conditions we impose for achieving a superlinear convergence rate, such as W-SRCQ and S-SOSC, do not even imply metric subregularity (see Example \ref{example:4}).
\end{remark}   

\begin{remark}\vio
It is worth noting that the results obtained in Section \ref{section:nonsingularity} and Section \ref{section:5} can be directly extended to the following generalized equation (GE): 
\begin{align}\label{eq:GE}
	0\in \phi(x)+\nabla g(x)^T\mN_D(g(x)),
\end{align}
where $\phi:\Rbb^n\to\Rbb^n$ is continuously differentiable, $g:\Rbb^n\to\Rbb^m$ is twice continuously differentiable and $D$ is the Cartesian product of a finite number of positive semidefinite cones and zeros.  Let $l(x,y):=\phi(x)+\nabla g(x)^Ty$ for $(x,y)\in \Rbb^n\times \Rbb^m$. Define the corresponding nonsmooth mapping $F:\Rbb^n\times \Rbb^m\to \Rbb^n\times \Rbb^m$ as
\begin{align}\nonumber
	F(x,y)=\begin{bmatrix}
		l(x,y)\\
		-g(x)+\Pi_D(g(x)+y)
	\end{bmatrix}.
\end{align}
For any $(x,y)\in \Rbb^n\times \Rbb^m$ (not necessarily a solution to \eqref{eq:GE}), define $$\widetilde\mC(x,y):=\mT_D(\Pi_D(g(x)+y))\cap (g(x)+y-\Pi_D(g(x)+y))^{\perp}$$ 
and
$$   \widetilde\sigma(x,y):=\sigma(g(x)+y-\Pi_D(g(x)+y),\mathcal{T}^2_{D}(\Pi_D(g(x)+y),\nabla g(x)^T\dxx)),$$
where $\Pi_D(\cdot)$ is the metric projection operator onto $D$, ${\cal T}_D(z)$ is the tangent cone to $D$ at $z$ defined in \eqref{eq:def-Tangent-g}, and ${\cal T}^2_D(z,d)$ is the second order tangent set to $D$ at $z$ in the direction $d$ defined in \eqref{eq:def-second-order-tangent}.  The constraint nondegeneracy for the GE \eqref{eq:GE} at $(x,y)$ is then defined as
\begin{equation}\label{eq:LICQ-GE}
	\nabla g(x)^T\Rbb^n+\lin(\widetilde\mC(x,y))=\mY.
\end{equation}
Analogously, the W-SRCQ for the GE \eqref{eq:GE} at $(x,y)$ is defined by
\begin{equation*}
	\nabla g(x)^T\Rbb^n+\aff(\widetilde\mC(x,y))=\mY.
\end{equation*}
We say the W-SOC for the GE \eqref{eq:GE}  holds at $(x,y)$ if and only if 
\begin{equation}\label{eq:W-SOC-GE}
	\la \dxx,\nabla_{x} l(x,y)\dxx\ra-\widetilde\sigma(x,y)\neq0 \quad \forall\, \dxx\in\appl(x,y)\backslash\{0\},
\end{equation}
where 
\begin{equation}\nonumber
	\appl(x,y):=\left\{\dxx\mid  \nabla g(x)^T\dxx\in\lin(\widetilde\mC(x,y))\right\}.
\end{equation}
Similarly, the S-SOC for the GE \eqref{eq:GE} at $(x,y)$ is defined correspondingly by
\begin{equation}\label{eq:S-SOC-GE}
	\la \dxx,\nabla_{x} l(x,y)\dxx\ra-\widetilde\sigma(x,y)\neq0 \quad \forall\, \dxx\in\app(x,y)\backslash\{0\},
\end{equation}
where 
\begin{equation}\nonumber
	\app(x,y):=\left\{\dxx\mid  \nabla g(x)^T\dxx\in\aff(\widetilde\mC(x,y))\right\}.
\end{equation}
Thus, the obtained results (e.g., Theorem \ref{theorem:bar-N-SDP} in Section \ref{section:nonsingularity} and Theorem \ref{theo:alg-convergence} in Section \ref{section:5}) can be shown by using exactly the same manners. {The noteworthy point is that, based on the proof of Theorem 1, the S-SOC of the generalized equation being ``$\neq0$'' instead of ``$>0$'' does not affect the validity of the conclusion.}

In the recent work \cite{gfrerer2021semismooth},  a generalized semismooth* Newton-like method for solving generalized equations is proposed.
When some appropriate regularity condition (see \cite[Theorem 4.4]{gfrerer2021semismooth} for details) is assumed, the method ensures local superlinear convergence. {The proposed regularity conditions are shown to be weaker than the metric regularity.} For \eqref{eq:GE} with $D$ being a polyhedron, the  regularity conditions hold under \cite[Assumptions 1 and 2]{gfrerer2021semismooth}, as elaborated  in \cite[Section 5]{gfrerer2021semismooth}. Moreover, as stated in \cite[Remark 5.2]{gfrerer2021semismooth}, for a solution  $(\overline x,\overline y)$ to \eqref{eq:GE},  \cite[Assumption 1]{gfrerer2021semismooth} is equivalent to the constraint nondegeneracy for the GE defined by \eqref{eq:LICQ-GE}. 
Additionally, when $D = \Rbb^n_{-}$, it is stated in \cite[Remark 5.10]{gfrerer2021semismooth} that \cite[Assumption 2]{gfrerer2021semismooth} is equivalent to the S-SOC \eqref{eq:S-SOC-GE}. However, it remains unclear under what circumstances the regularity conditions proposed in \cite[Theorem 4.4]{gfrerer2021semismooth} will hold for non-polyhedral sets, such as positive semidefinite cones.
Interestingly, in \cite[Example 5.13]{gfrerer2021semismooth}, their algorithm produces the same sequence as the classical semismooth Newton method using $\mU_\mI$ (i.e., Algorithm \ref{algorithm:ASN} without correction). For this particular case, the correction map ${\cal P}_{\delta}(\cdot)$, defined by \eqref{eq:def-pi-mapping}, reduces to the identity map.

\end{remark}

\section{Numerical results}\label{section:numerical}

	In this section, we present a series of examples to empirically validate our theoretical findings and demonstrate the fast local convergence of our proposed algorithm. In our numerical experiments, we consider a range of problems, including both convex and nonconvex nonlinear SDPs. 
		It's worth noting that the generalized Jacobian regularity condition does not hold for all the test instances, thereby rendering classic convergence results for semismooth Newton methods inapplicable. 
	
	As preliminary numerical results, we solve the Newton linear system involved in each iteration of Algorithm \ref{algorithm:ASN} exactly, i.e., the parameter $\eta$ is set to be zero. Following Remark \ref{rmk:termination}, we assess the accuracy of an approximate solution $Z$ to the NLSDP \eqref{prog:NLSDP} by measuring the residual $\|F(Z)\|$. For a given stopping tolerance $\epsilon >0$, our algorithm terminates when $\|F(Z)\|<\epsilon$. Throughout our experiments, we fix $\epsilon = 10^{-10}$. All computational results are generated using MATLAB (version 9.10) running on a Mac mini (Apple M1, 16 G RAM).

	While our focus has been on investigating the local convergence properties of our newly proposed semismooth Newton method, our numerical tests indicate that the algorithm is quite robust to the selection of the starting point. To facilitate reproducibility, we adhere to the following procedures for generating the starting point:
	\begin{itemize}
		\item[(i)] Fix the random seed and select the perturbation constant $C_{\rm perturb}$. In our tests, we fix the random seed via {\tt rng(2,``twister'')} and set $C_{\rm perturb}=10$ for examples \ref{example:3}, \ref{example:4} and \ref{example:5}, and set $C_{\rm perturb}=1$ for example \ref{example:6}. 
		\item[(ii)] Generate random noises $ (\Delta x, \Delta \xi, \Delta \Gamma)$ with each entry following a standard normal distribution $N(0,1)$, e.g.,  
		\[{\tt  \Delta \Gamma = randn(n),  \quad \Delta \Gamma = 0.5*(\Delta \Gamma + \Delta \Gamma')}.\]
		\item [(iii)] Set the starting point $\widehat Z^0$ as
		\[
			\widehat Z^0 = \overline Z + C_{\rm perturb} \times \left(\frac{\Delta x}{\|\Delta x\|}, \frac{\Delta \xi}{\|\Delta \xi\|}, \frac{\Delta \Gamma}{\|\Delta \Gamma\|}\right),
		\]
		where $\overline Z$ is the known optimal solution to \eqref{prog:NLSDP}.
	\end{itemize}
	It is important to note that the perturbation constant $C_{\text{perturb}}$ employed is not insignificant. Remarkably, varying the random seeds and the perturbation constant $C_{\text{perturb}}$ results in comparable quadratic convergence behavior.

\begin{example}\label{example:3}
	Consider the following convex quadratic SDP problem:
	\begin{align}\label{eq:prob-example3}
		\begin{aligned}
			\min\quad        &\frac{1}{2}\| x_{11}-1\| ^2+\frac{1}{2}\| x_{22}-x_{12}-x_{21}\| ^2 \\
			\mbox{s.t.}\quad &  \la E,X\ra\le 1,\\
			&X\in \Sbb^2_+ , 
		\end{aligned}
	\end{align}
	where $X=\begin{bmatrix}x_{11} &x_{12}\\x_{21}&x_{22}\end{bmatrix}\in \Sbb^2$. It is not difficult to observe that $\overline X=\begin{bmatrix} 1 &0\\0&0\end{bmatrix}$ is an optimal solution and $(\overline \xi,\overline\Gamma)=(0,0)$ is the unique multiplier. For such $\overline{Z}=(\overline X,\overline \xi,\overline\Gamma)$, 
	one can check that the W-SOC and the constraint nondegeneracy hold at $\overline{Z}$, but the S-SOSC fails. 
	Then, Theorem \ref{theorem:bar-N-SDP} implies the nonsingularity of $\overline {\cal U}_{0}$, which can also be verified directly by noting that the linear system
	\begin{align*}
	\overline {\cal U}_{0}(\dZZ)=
	\begin{bmatrix}
	&\begin{bmatrix}\dxx_{11} &-\dxx_{22}+\dxx_{12}+\dxx_{21}\\ -\dxx_{22}+\dxx_{12}+\dxx_{21}&\dxx_{22}-\dxx_{12}-\dxx_{21}\end{bmatrix}+\dxi E+\dGG \\
	&\la E,\Delta X\ra  \\
	&\begin{bmatrix}\dGG_{11} &\dGG_{12}\\ \dGG_{21}&-\dxx_{22}\end{bmatrix}
	\end{bmatrix}
	=0
	\end{align*}
	with $\dZZ=(\dXX,\dxi,\dGG)=\left(\begin{bmatrix}\dxx_{11} &\dxx_{12}\\\dxx_{21}&\dxx_{22}\end{bmatrix},\dxi,\begin{bmatrix}\dGG_{11} &\dGG_{12}\\ \dGG_{21}&\dGG_{22}\end{bmatrix}\right)\in \Sbb^2\times\RR\times\Sbb^2$ has only trivial zero solution $\dZZ=0$.	
	 Thus, Algorithm \ref{algorithm:ASN} using ${\cal U}_{0}$ can be applied to solve \eqref{eq:prob-example3}. In the test, we set the parameter $\delta = 0.5$, which lies between $0$ and the smallest nonzero magnitude eigenvalue of $\overline{A}$.
	 We report the detailed performance of our algorithm in Figure \ref{fig:ex3} and Table \ref{tab:ex3}.
	 \begin{table}[htbp]
		\centering
		\begin{tabular}{cccc}
			\hline \\[-5pt]
			 Iteration $k$ & $\|F(Z^k)\|$ & $\|Z^k - \overline{Z}\|$ & $\sigma_{\min}({\cal U}^k)$ \\
			\hline
			 0 & 2.51e+01 & 1.43e+01 & 4.03e-01 \\
			 1 & 7.50e-01 & 9.36e-01 & 3.07e-01 \\
			 2 & 4.19e-02 & 2.10e-02 & 3.04e-01 \\
			 3 & 3.33e-05 & 2.88e-05 & 3.04e-01 \\
			 4 & 2.06e-14 & 1.80e-14 & 3.04e-01 \\
			\hline
			\end{tabular}
		\caption{Convergence details of Algorithm \ref{algorithm:ASN} for solving \eqref{eq:prob-example3}.}
		\label{tab:ex3}
		\end{table}
	In the table, we report residuals $\|F(Z^k)\|$ and $\|Z^k - \overline{Z}\|$, and the minimum singular value of ${\cal U}^k$, denoted by $\sigma_{\min}({\cal U}^k)$. 
	The detailed results in Table \ref{tab:ex3} demonstrate the quadratic convergence of Algorithm \ref{algorithm:ASN}, the locally uniform boundedness of $({\cal U}^k)^{-1}$, and the relation $F(Z^k) = \Theta(\|Z^k - \overline{Z}\|)$ asserted in Remark \ref{rmk:termination}.

	 \begin{figure}[htbp]
		\centering
		\includegraphics[width=0.72\textwidth]{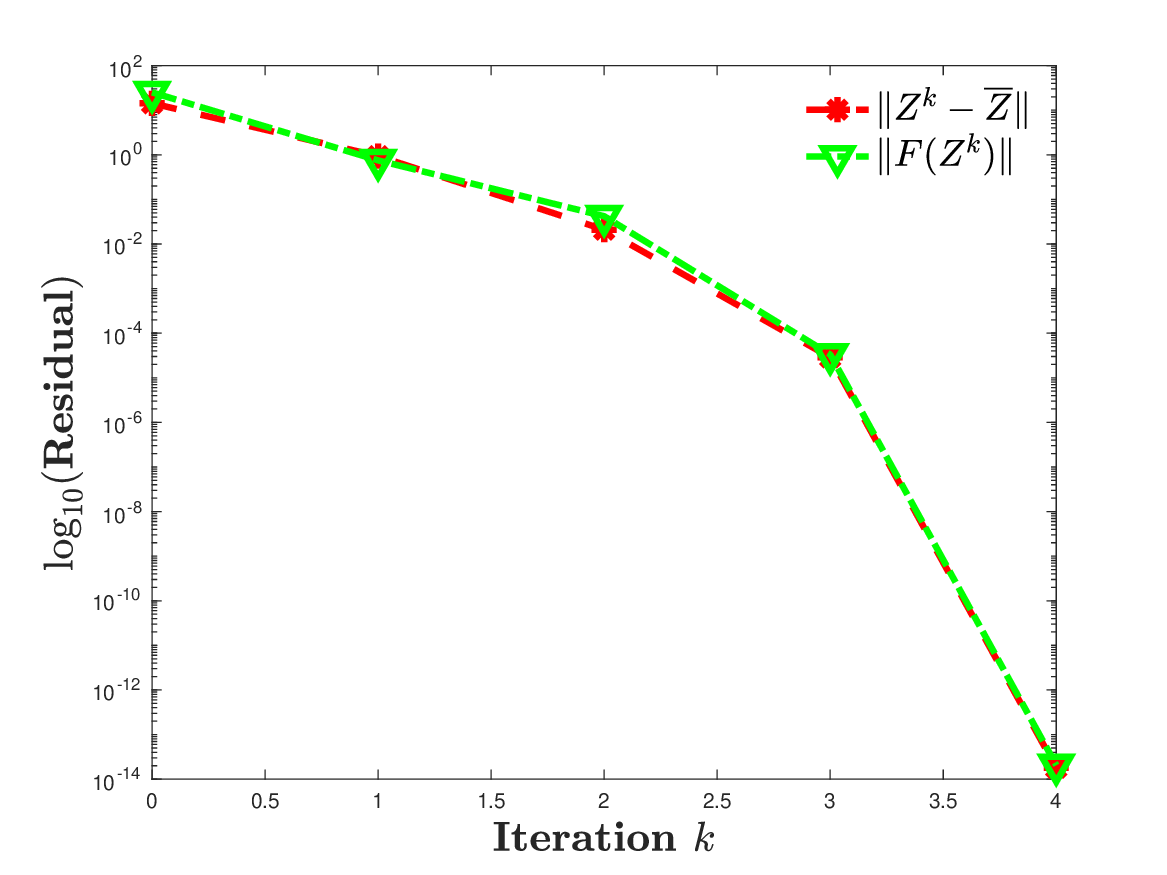}
		\caption{Quadratic convergence of Algorithm \ref{algorithm:ASN} for solving \eqref{eq:prob-example3}.}
		\label{fig:ex3}
	\end{figure}
\end{example}

The following example illustrates the efficiency of our algorithm in solving nonlinear SDPs even in the absence of calmness, as defined in \cite[Definition 9(30)]{rockafellar2009variational}.

\begin{example}[Example 3 in \cite{ding2017characterization}; see also \cite{zhou2017unified}] \label{example:4}
	Let $B=\begin{bmatrix}3/2 & -2\\ -2 & 3\end{bmatrix}$ and $b=B^{-1/2}\begin{bmatrix}5/2 \\ -1\end{bmatrix}$. Consider the following convex quadratic SDP problem:
\begin{align}\label{eq:prob_example4_P}
\begin{aligned}
\min\quad        &\frac{1}{2}\| x+b\| ^2+t \\
\mbox{s.t.}\quad &  \mA^* x+tE+I+\Delta\in\Sbb^2_+ , \\
&t\ge 0,
\end{aligned}
\end{align}
and its dual
\begin{align}\label{eq:prob_example4_D}
\begin{aligned}
\max\quad        &-\frac{1}{2}\| \mA Y-b\| ^2-\la I+\Delta,Y\ra \\
\mbox{s.t.}\quad &  \la E,Y\ra \le 1 , \\
&Y\in \Sbb^2_+,
\end{aligned}
\end{align}
where $\mA Y=B^{1/2}\Diago (Y)$ for all $Y\in\Sbb^2$, $\mA^* x=\Diago(B^{1/2}x)$ for all $x\in\RR^2$ and $\Delta=\Diago(-\varepsilon,\varepsilon)$ with $\varepsilon\ge 0$. As shown in \cite{ding2017characterization}, the solution map of the KKT system associated with problems \eqref{eq:prob_example4_P} and \eqref{eq:prob_example4_D}, denoted by ${S}_{\rm KKT}$ in \cite{ding2017characterization}, is not calm with respect to the perturbation $\varepsilon$. When $\varepsilon =0$, it is not difficult to see that $\overline Y=\begin{bmatrix}1& 0 \\ 0 & 0\end{bmatrix}$ is the unique optimal solution to \eqref{eq:prob_example4_D} with the unique multiplier $(0,0)\in \mathbb{R}\times\mathbb{S}^2$. Then, one can check that the W-SOC and constraint nondegeneracy hold at $\overline{Z}=(\overline Y, 0,0)$, but the S-SOSC fails. From Theorem \ref{theorem:bar-N-SDP} and Proposition \ref{prop:pd:LSDP-V01}, we know that $\overline {\cal U}_{0}$ is nonsingular but $\overline {\cal U}_{\cal I}$ is singular. Thus, we can apply Algorithm \ref{algorithm:ASN} with ${\cal U}_{0}$ to solve  \eqref{eq:prob_example4_D}. In the test, we set parameter $\delta = 0.5$. The convergence details of our algorithm are reported in Table \ref{tab:ex4} and Figure \ref{fig:ex4}. Again, the quadratic convergence of our algorithm is confirmed.

\begin{figure}[htbp]
	\centering
	\includegraphics[width=0.72\textwidth]{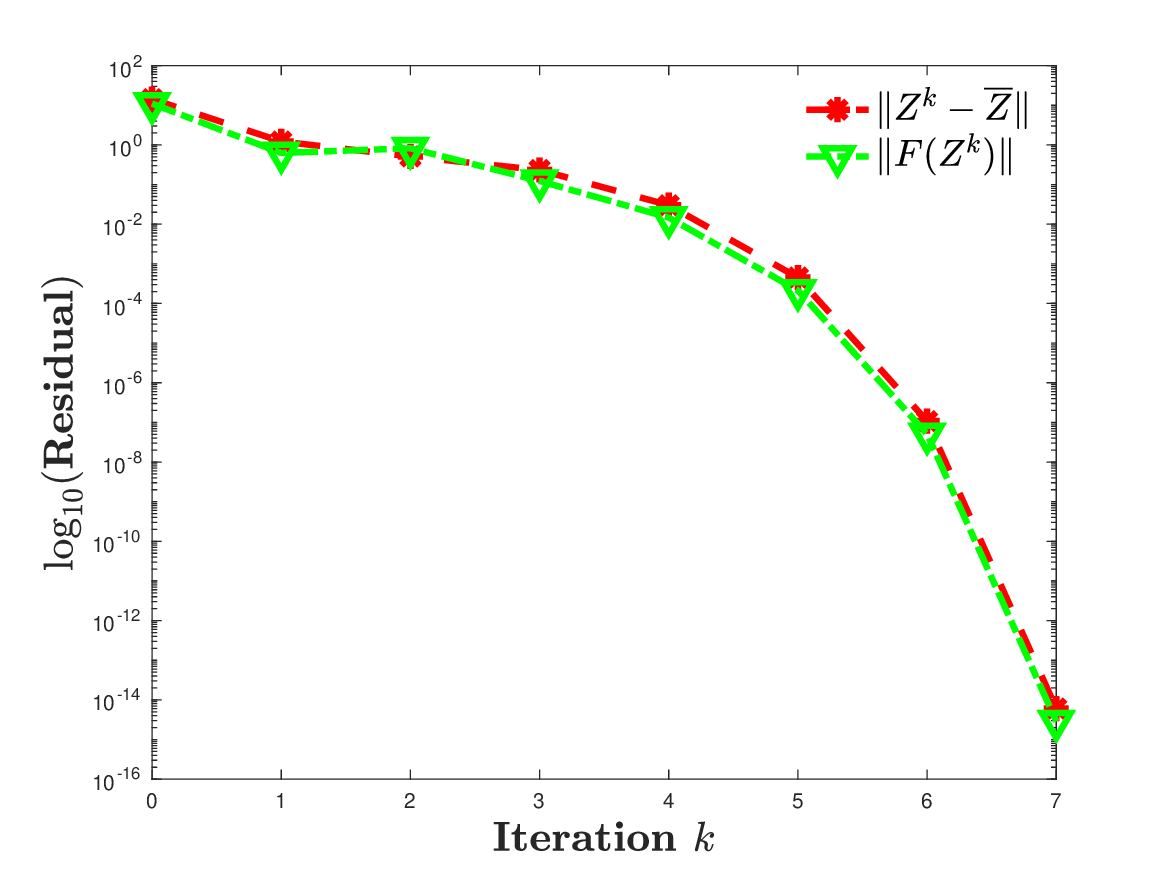}
	\caption{Quadratic convergence of Algorithm \ref{algorithm:ASN} for solving \eqref{eq:prob_example4_D}.}
	\label{fig:ex4}
\end{figure}
\begin{table}[htbp]
	\centering
	\begin{tabular}{cccc}
		\hline \\[-5pt]
		 Iteration $k$ & $\|F(Z^k)\|$ & $\|Z^k - \overline{Z}\|$ & $\sigma_{\min}({\cal U}^k)$ \\
		\hline
		0 & 1.10e+01 & 1.43e+01 & 4.16e-02 \\
		1 & 6.26e-01 & 1.22e+00 & 1.48e-01 \\
		2 & 8.23e-01 & 5.12e-01 & 1.38e-01 \\
		3 & 1.23e-01 & 2.33e-01 & 3.70e-01 \\
		4 & 1.46e-02 & 3.02e-02 & 2.85e-01 \\
		5 & 2.13e-04 & 4.51e-04 & 2.76e-01 \\
		6 & 5.08e-08 & 1.08e-07 & 2.76e-01 \\
		7 & 2.89e-15 & 6.05e-15 & 2.76e-01 \\
		\hline
		\end{tabular}
	\caption{Convergence details of Algorithm \ref{algorithm:ASN} for solving \eqref{eq:prob_example4_D}.}
	\label{tab:ex4}
	\end{table}
\end{example}

In the next two examples, we evaluate the performance of our algorithm for solving medium size nonlinear SDPs.

\begin{example}\label{example:5}
Consider the following convex quadratic SDP problem:
\begin{align}\label{eq:prob_example5}
\begin{aligned}
\min\quad        &\frac{1}{2}\| X_{11}-I\| ^2+\frac{1}{2}\| X_{12}\| ^2+\frac{1}{2}\| X_{21}\| ^2 \\
\mbox{s.t.}\quad &  X\in \Sbb^n_+ , 
\end{aligned}
\end{align}
where $X=\begin{bmatrix} X_{11} &X_{12}\\X_{21}&X_{22}\end{bmatrix}$ is the block representation with $X_{11}\in\RR^{l_1\times l_1}$ and $X_{22}\in\RR^{l_2\times l_2}$. We know that $\overline X=\begin{bmatrix} I &0\\0&0\end{bmatrix}$ is an optimal solution and $\overline\Gamma=0$ is the unique multiplier. One can check that the W-SOC and the constraint nondegeneracy holds at $\overline{Z}=(\overline X,\overline\Gamma)$, but the SOSC fails since $\overline X$ is not isolated. Then, we have from Theorem \ref{theorem:bar-N-SDP} that $\overline {\cal U}_{0}$ is nonsingular and Algorithm \ref{algorithm:ASN} with ${\cal U}_{0}$ can be applied to solve \eqref{eq:prob_example5}. Here, the test instance is constructed by setting $l_1=60$, $l_2=40$, and the parameter  $\delta$ is set to $0.5$. Table Figure \ref{fig:ex5} shows our algorithm converges rapidly. 

\begin{table}[htbp]
	\centering
	\begin{tabular}{cccc}
		\hline \\[-5pt]
		 Iteration $k$ & $\|F(Z^k)\|$ & $\|Z^k - \overline{Z}\|$ & $\sigma_{\min}({\cal U}^k)$ \\
		\hline
		0 & 8.39e-01 & 7.06e-01 & 6.02e-01 \\
 1 & 3.54e-02 & 3.54e-02 & 6.18e-01 \\
 2 & 1.85e-14 & 1.75e-14 & 6.18e-01 \\
		\hline
		\end{tabular}
	\caption{Convergence details of Algorithm \ref{algorithm:ASN} for solving \eqref{eq:prob_example5}.}
	\label{tab:ex5}
	\end{table}

\begin{figure}[htbp]
 \centering
 \includegraphics[width=0.72\textwidth]{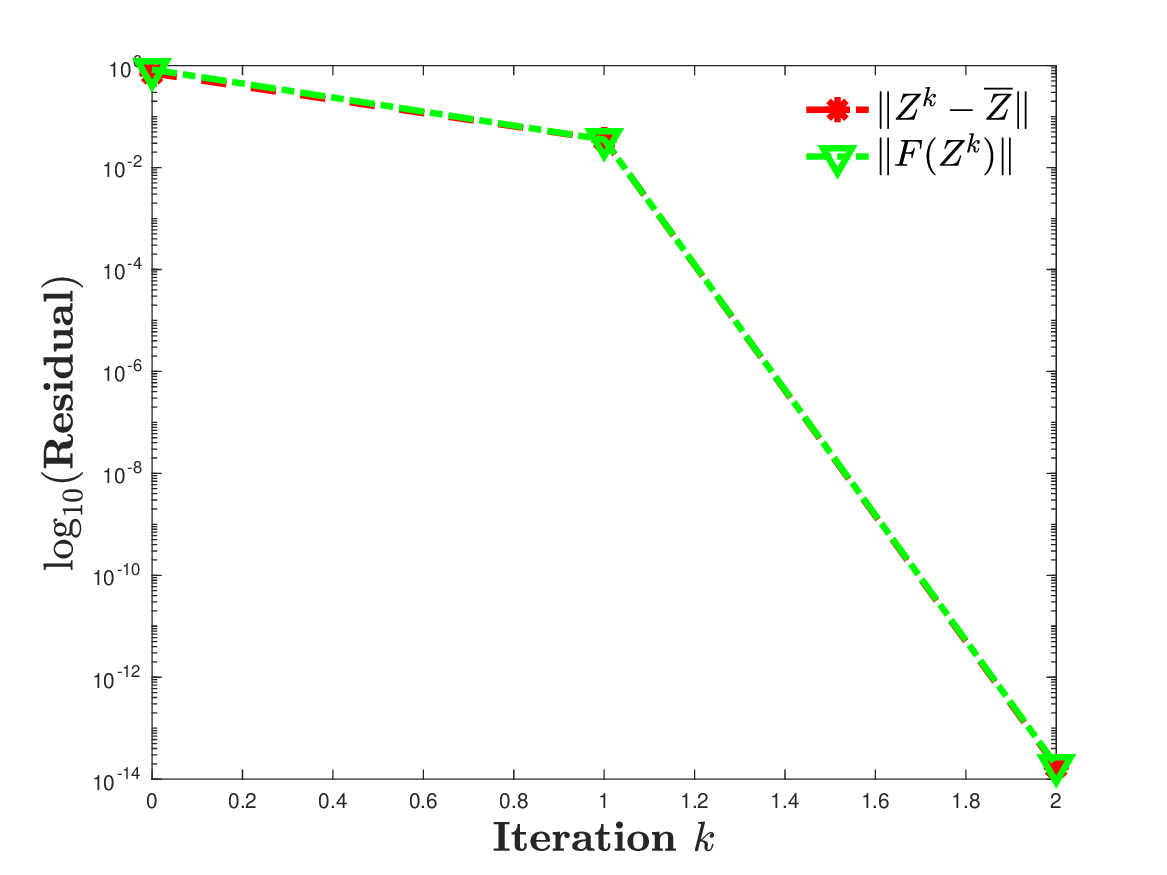}
	\caption{Quadratic convergence of Algorithm \ref{algorithm:ASN} for solving \eqref{eq:prob_example5}.}
 \label{fig:ex5}
\end{figure} 
\end{example} 

\begin{example}\label{example:6}
Recall problem \eqref{eq:prob_example1} in Example \ref{example:1}. Note that problem  \eqref{eq:prob_example1} is nonconvex, and $\overline X = 0$ is a local optimal solution with unbounded Lagrange multiplier set $M(\overline X)$. It can be verified that the S-SOSC holds at $\overline Z= (\overline{X},\overline \xi_{12},\overline\xi_{22},\overline{\Gamma})= (0,0,0,0) \in \mathbb{S}^n \times \mathbb{R}^{l_1\times l_2}\times\mathbb{R}^{l_2\times l_2}\times \mathbb{S}^n$ and the W-SRCQ holds at $\overline X$ with respect to $(\overline \xi_{12},\overline\xi_{22},\overline{\Gamma})$. Then, it is known from Theorem \ref{theorem:bar-N-SDP} that $\overline{\cal U}_{\cal I}$ is nonsingular. 
 Thus, Algorithm \ref{algorithm:ASN} with ${\cal U}_{\cal I}$ can be applied to solve \eqref{eq:prob_example1}. In our test, we set the parameter $\delta = 0.5$, and  $l_1=60$ and $l_2=40$. The detailed performance of our algorithm is reported in Table \ref{tab:ex6}. As one can observe, our algorithm converges in one iteration from a randomly generated starting point with initial distance $\|Z^0 - \overline{Z}\| = 1.73$.  

 \begin{table}[htbp]
	\centering
	\begin{tabular}{cccc}
		\hline \\[-5pt]
		 Iteration $k$ & $\|F(Z^k)\|$ & $\|Z^k - \overline{Z}\|$ & $\sigma_{\min}({\cal U}^k)$ \\
		\hline
		0 & 1.85e+00 & 1.73e+00 & 3.27e-01 \\
 1 & 2.98e-15 & 4.18e-15 & 3.27e-01 \\
		\hline
		\end{tabular}
	\caption{Convergence details of Algorithm \ref{algorithm:ASN} for solving \eqref{eq:prob_example1}.}
	\label{tab:ex6}
	\end{table}
\end{example}

\section{Conclusions}\label{section:conclusion}
	In this paper, we explore the weak second-order condition (W-SOC) and the weak strict Robinson constraint qualification (W-SRCQ) within the framework of general optimization problems \eqref{prog:P}. We establish the sufficient and necessary conditions for the existence of a nonsingular element in the generalized Jacobians of the KKT residual map $F$ at a given KKT point. Moreover, we uncover a profound primal-dual connection between the W-SOC and the W-SRCQ for the convex quadratic SDP. Drawing upon these theoretical advancements, we propose a novel semismooth Newton method with a correction step, ensuring the local convergence and yielding a (quadratic) superlinear convergence rate without relying on the classic generalized Jacobian regularity. 
There are still many issues that deserve further exploration. For example, addressing the globalization of the proposed algorithm and developing highly efficient implementations present significant challenges. Additionally, extending our theoretical findings to broader sets is highly desirable. Our results could potentially be extended to variational inequalities with semidefinite constraints and optimization problems with matrix cone constraints, as defined in \cite{ding2014introduction}. Furthermore, exploring the generalization of our results to more general sets, such as ${\cal C}^2$-cone reducible sets, would be another challenging topic.

\begin{acknowledgements}
	We would like to thank the referees as well as the   editors for their constructive comments that have helped to improve the quality of the paper.
\end{acknowledgements}

\bibliographystyle{spmpsci}
\bibliography{sn-bibliography}

\end{document}